\documentclass[11pt]{amsart}

\usepackage{amsmath, amsfonts, amssymb}
\usepackage{enumerate, xspace}
\usepackage{tikz}
\usetikzlibrary{decorations.markings}

 \let\mathds\mathbf
\usepackage[hidelinks]{hyperref}

\newcommand{\ignore}[1]{}

\newcommand{\Pene}{P\`ene}

\newcommand{\dimX}{D}

\newcommand{\tphi}{\widetilde{\phi}}

\newcommand{\betamax}{\beta_{\operatorname{max}}}

\newcommand{\Lip}{\operatorname{Lip}}

\newcommand{\Jmax}{J_{\operatorname{max}}}

\newcommand{\one}{{\mathds{1}}}
\newcommand{\onebr}[1]{\one_{\{#1\}}}

\newcommand\dto{\xrightarrow{d}}

\let\epsilon\varepsilon

\let\phi\phi

\let\tilde\widetilde

\let\bar\widebar 

\numberwithin{equation}{section}

\newtheorem{thm}{Theorem}[section]

\newtheorem{cor}[thm]{Corollary}

\newtheorem{prop}[thm]{Proposition}

\newtheorem{defn}[thm]{Definition}

\newtheorem{rmk}[thm]{Remark}

\newcommand\bE{{\mathbb E}}

\newcommand\bN{{\mathbb N}}

\newcommand{\Z}{\mathbb{Z}}
\newcommand{\R}{\mathbb{R}}

\begin{document}

\date{\today.}

\title[Stable laws for heavy-tailed  observables]{Stable laws for heavy-tailed  observables on polynomially mixing billiards.}

\author[M. Nicol]{Matthew Nicol$^*$}
\thanks{* Corresponding author}
\address{Matthew Nicol\\ Department of Mathematics\\
University of Houston\\
Houston\\
TX 77204\\
USA}
\email{matthewjames.nicol@gmail.com}
\urladdr{https://sites.google.com/cougarnet.uh.edu/nicol/home}

\author[Manpreet Singh]{Manpreet Singh}
\address{Manpreet Singh\\ Department of Mathematics \& Physics\\
Colorado State University Pueblo\\
Pueblo\\
CO 81001-4901\\
USA}
\email{manpreet.singh@csupueblo.edu}

\author[A. T\"or\"ok]{Andrew T\"or\"ok}
\address{Andrew T\"or\"ok\\ Department of Mathematics\\
  University of Houston\\
  Houston\\
  TX 77204\\
  USA and 
{Institute of Mathematics of the Romanian Academy, Bucharest, Romania.}}
\email{torok@math.uh.edu}
\urladdr{http://www.math.uh.edu/~torok/}

\thanks{ MN thanks the Leverhulme Foundation and Warwick University for support, and 
the Mathematics Research Centre, University of Warwick, for their hospitality.  MN was also supported in
  part by NSF Grant DMS 2009923. AT was supported in part by NSF Grant DMS 1816315. }

\renewcommand{\subjclassname}{2020 Mathematics Subject Classification}


\begin{abstract}

We investigate the competition between two mechanisms that produce stable
laws in deterministic dynamical systems: heavy tails of the observable,
sometimes called the Noah effect, and slow mixing of the system, sometimes
called the Joseph effect. We first consider dispersing billiards with cusps for
which suitably scaled H\"older observables satisfy a stable law of index
$1/\gamma$. Let
\[
    \phi(x)=d(x,x_0)^{-2/\alpha},
\]
where $0<\alpha<2$, $\alpha\not =1$, and $x_0$ is a generic point away from the cusps. We prove a
transition at $1/\alpha=\gamma$. If $1/\alpha>\gamma$, the normalized Birkhoff
sums converge to a stable law of index $\alpha$, determined by the tail of
$\phi$. If $1/\alpha<\gamma$ and the cusp contribution does not vanish, the
limit has index $1/\gamma$ and is determined by the dynamics. We also give
conditions under which an induced $\alpha$-stable law lifts when the cusp
contribution vanishes. The proof decomposes the observable into components
supported near and away from the cusps and compares their Birkhoff sums on the
two competing scales. For $1<\alpha<2$, point-process convergence alone does
not imply convergence of the centered sums. We establish the required
vanishing-small-values condition for distance observables on exponentially
mixing Young towers and prove a general maximal negligibility estimate used in
the lifting argument. As a consequence, we complete the stable-law results for
heavy-tailed observables on Liverani--Saussol--Vaienti maps for
$0<\alpha<2$, $\alpha\ne1$, and $0<\gamma<1$. When the observable is singular
at the indifferent fixed point, the two mechanisms combine and the stable index
is $(\frac{1}{\alpha}+\gamma)^{-1}$, rather than the index $\alpha(1-\gamma)$ suggested
by the tail of the observable. Finally we give evidence that the addition formula
$\bigl(\frac1\alpha+\gamma\bigr)^{-1}$ is not general, but depends upon how close
orbits approach the cusp or indifferent fixed point before returning to the
inducing set.

    \end{abstract}

\maketitle
\newpage 
\setcounter{tocdepth}{3}
\tableofcontents

\section{Introduction}\label{sec:intro}

\bigskip

There are two main mechanisms for producing stable limit laws and anomalous diffusion in 
physical systems,  heavy-tailed observables or slow mixing (strong dependence)  of the system. The first effect, heavy
tails, is sometimes called the Noah effect~\cite[Section 4.5]{Whitt} while slow mixing is sometimes
called the Joseph effect~\cite[Section 4.6]{Whitt}.  The terminology is due to Mandelbrot~\cite{Mandelbrot_Wallis}
and motivated by the deluge-like aspect of a heavy-tailed observable and the predictability of a slowly mixing system.
Whitt writes ``the most complicated case involves both heavy tails and dependence. Unfortunately there is not
a well-developed theory for stochastic process limits in this case''~\cite[Section 4.7]{Whitt}. In this paper we
investigate the interplay of both effects in polynomially mixing billiard systems with cusps and intermittent-type maps.

 On these  canonical models
of anomalous diffusion, regular (H\"older or Lipschitz) observables satisfy stable laws due to the strong dependence.
We consider a heavy-tailed observable
$\phi(x)=d(x,x_0)^{-\frac{2}{\alpha}}$, $0<\alpha <2$ ($\alpha \not=1$), on a polynomially
mixing billiard $(T,\tilde{Q},\mu)$ as described
in~\cite{Zhang2017,Jung_Zhang2018}. Since the invariant measure $\mu$ is
equivalent to Lebesgue and the phase space is two-dimensional, the
observable $\phi$ has a tail distribution of index $\alpha$, i.e.
$\mu (\phi >t)\sim t^{-\alpha}$. The billiard systems we consider have a
sufficiently slow mixing rate that H\"older observables satisfy a stable
law of index $1/\gamma$. Billiards with this property have been studied in
a series of recent papers~\cite{Zhang2017,Jung_Zhang2018,Melbourne_Varandas2020, Jung_Pene_Zhang2020}. We
focus in this work on stable laws, rather than convergence in the $J_1$ or
$M_1$ topologies. 

Theorems~\ref{main1} and~\ref{main2} establish a  transition in terms of $\alpha$ and $\gamma$ in that the stable law index arises from the heavy-tailed observable if $\frac{1}{\alpha}>\gamma$ and from the 
slow mixing mechanism if $\gamma>\frac{1}{\alpha}$.

 We have kept the notation from~\cite{CNT2025} as much as
possible to allow comparison with the results and ideas of that paper. We
 induce on a set $M$ such that the cusps lie in the complement of $M$.
We assume $x_0\in M$, so that in particular $x_0$ is not a cusp
point. The set $M$ is such that the first return map $F:M\to M$ can be
modeled as a Young Tower with exponential tails. This means that there is a
subset $M_0 \subset M$ such that the return map to $M_0$ under $F$ is
uniformly hyperbolic with singularities and exponentially mixing for
H\"older observables. By contrast, the return time function $R$ to $M$ has
polynomial tails. This
setup is described in Zhang's paper~\cite{Zhang2017}, which gives a class
of billiard systems (a form of Machta billiard~\cite{Machta_83}) with
arbitrarily slow mixing rates. The initial work on analyzing mixing rates
of Machta type billiards was carried out in~\cite{CZ2005}, where the rate
of roughly $O(1/n)$ was established.

The key approach we take is to decompose $\phi$ into two parts; we write
\[
  \text{$\phi=\phi_1+\phi_2$ where $\phi_1=\phi \one_{M^c}$ and
    $\phi_2=\phi \one_{M}$.}
\]
If the point $x_0 \in M$ then $\phi_1$ fits into the class of observables
considered by~\cite{Zhang2017,Jung_Zhang2018,Jung_Pene_Zhang2020} and we
use results from  their work.
Although this  decomposition isolates the two mechanisms the resulting limit law is not a
 formal consequence of the two component limit laws.  The components  have Birkhoff sums which fluctuate on 
 competing scales and the induced stable laws must be lifted through return times whose
 own scaled Birkhoff sums satisfy stable laws. 
 
 Theorem~\ref{distribution} and its Corollary~\ref{distribution-corollary}
 show that in an exponentially mixing system Birkhoff sums of a general  observable with tail index $\alpha$
 converge to zero in distribution if scaled by a factor $n^{-\frac{1}{\alpha} -\epsilon}$ for any $\epsilon>0$. This result  is useful to allow lifting 
 of an induced observable to the ambient system when the usual techniques, such as the continuous mapping theorem, are not applicable.

The function $\phi_2$ is a heavy tailed observable which is a function of
distance, maximized at a point of an exponentially mixing billiard system.
Hence it fits into the class of observables considered by \Pene{} and
Saussol~\cite{Pene-Saussol2020}. A consequence of their work is  that suitably scaled point
processes of $\phi_2$,
\[
  N_n:=\sum_{j=1}^n \delta_{(j/n, \phi_2\circ T^j/b_n)}
\]
converge in distribution to a compound or simple Poisson Point process $N$ on
$[0,1] \times \R$.   In the case
$\alpha \in (1,2)$ their work does not imply
a stable law for scaled Birkhoff sums. To achieve this we   establish a ``vanishing small values'' condition for
generic $x_0$ in the setting of exponentially mixing Young Towers, Theorem~\ref{thm:mixing-condition_Young-Tower}. 
We note that if $\alpha \in (0,1)$ then  the observable satisfies a stable law by~\cite{DH95}.

As a consequence of Theorem~\ref{distribution} we  answer the  open questions from ~\cite{CNT2025} where an intermittent map
$(T_{\gamma}, [0,1],\mu_{\gamma})$ was considered.
We show that a stable limit law holds for the observable $\phi(x)=d(x,x_0)^{-\frac{1}{\alpha}}$ ($x_0\not=0$) in all parameter ranges
(save for $\alpha=1$) extending the results from the regime
$\alpha<1-\frac{1}{\gamma}+\frac{1}{\gamma^2}$. These results are given in
Proposition~\ref{thm:intermittent}.

In this paper we determine the stable law that holds when both effects combine.  In Theorem~\ref{cusp_intermittent} we
give the stable law satisfied by the observable
$\phi(x)=x^{\frac{-1}{\alpha}}+\kappa$ ($\kappa$ chosen so
that $\phi$ is centered), which is maximized at the
indifferent fixed point.  The stable law has index  $\bigl(\frac1\alpha+\gamma\bigr)^{-1}$.  Note that this is
not the index suggested by the tail of $\phi$, namely
$\alpha(1-\gamma)$. This extends results of  Gou\'ezel~\cite[Theorem
1.3 (3)]{Gouezel_Intermittent} from the range $\frac{1}{\alpha}+\gamma<1$
to all $0<\gamma<1$, $\alpha\in (0,2)$ ($\alpha \not =1$).  In the discussion section we describe work in progress indicating that this addition rule is not universal.

About the structure of the paper: the main results are presented in
Section~\ref{sec:main-results}: polynomially mixing billiards in
\S\ref{sec:billiards-with-cusps} and intermitent maps in
\S\ref{sec:intermittent-maps}.
The lifting arguments are discussed in
Section~\ref{sec:lifting-continuous-mapping}. The results pertaining to
billiards are proven in
Sections~\ref{sec:tower-construction}--\ref{sec:proof-main}. The Appendix,
Section~\ref{sec.appendix}, provides abstract lifting tools and some of the
proofs.

A more detailed summary is at the end of Section~\ref{sec:main-results}.


\section{Background}\label{sec.background}

\subsection{Regularly varying functions and domains of
  attraction}\label{ssec:regu}

For background on the relation between domains of attraction of stable laws
and regularly varying functions we refer to Feller~\cite{Fel71} or Bingham,
Goldie and Teugels~\cite{Bingham-Goldie-Teugels-1987}.

For Machta type billiards $\phi(x):=d(x,x_0)^{-\frac{2}{\alpha}}$, $0<\alpha<2$,
satisfies $\mu(\phi>t)\sim t^{-\alpha}$ and is regularly varying with
stable index $\alpha$.

For  regularly varying $\phi$ on a probability space $(\Omega,\mu)$ we define scaling
constants $b_n$ (related to the index) and $c_n$ (centering) by

\begin{defn}\label{def:scaling-constants}\ \null
  \begin{enumerate}[--]

  \item the scaling sequence $(b_n\sim n^{\frac{1}{\alpha}}) $ satisfies
    \begin{equation}\label{eq:tail1}
      \lim_{n \to \infty} n \mu (|\phi| > b_n ) = 1.
    \end{equation}

  \item the centering sequence $(c_n)_{n\ge 1}$ is given by
    \begin{equation}\label{eq:centering}
      c_n =
      \begin{cases}
        0 &\text{if $\alpha \in (0,1)$}\cr n 
        \mu(\phi)& \text{if $\alpha \in(1,2)$}
      \end{cases}.
    \end{equation}

  \end{enumerate}
\end{defn}

\begin{rmk}\label{rmk:usual-stable-description}

  A random variable  $X$ is in the domain of attraction of a stable law of index
  $\alpha\in(0,2)$ if
  \begin{equation}\label{eq:slow-variation1}
    {\mu(|X|> x)} =x^{-\alpha} L(x)
  \end{equation}
  (that is, its tail probability is regularly varying) and
  \begin{equation}\label{eq:slow-variation2}
    \lim_{x\to\infty} \frac{\mu(X > x)}{\mu(|X|> x)} = p
  \end{equation}
  for a slowly varying function $L(x)$ and $0\le p \le 1$. See
  \cite[Introduction]{DH95}.
\end{rmk}

We recall some estimates for regularly varying functions.

\begin{prop}[Karamata,
  \cite{Bingham-Goldie-Teugels-1987}]\label{prop:karamata}

  Let $\phi$ be regularly varying with index $\alpha \in (0,2)$. 

    The following hold for all $\epsilon>0$:\begin{enumerate}[(a)]

\item
  $\displaystyle \lim_{n \to \infty} n \mu ( |\phi| > \epsilon b_n ) =
  \epsilon^{-\alpha}$ (from the definition of $b_n$ and the regular
  variation of $\phi$)

\item If $k > \alpha$ then
  \[
    \displaystyle \mu(|\phi|^k \onebr{|\phi|\le u}) \sim
    \frac{\alpha}{k-\alpha}u^k \mu (|\phi|>u) \text{ as $u\to\infty$}
\]
In particular:

\item if $\alpha\in(0,2)$ then
\[
  \mu(| \phi|^2 \one_{\left\{ | \phi | \le \epsilon b_n \right\}}) \sim
  \frac{\alpha}{2 - \alpha} (\epsilon b_n)^2 \mu(| \phi | > \epsilon b_n)
\]

\item if $\alpha \in (0,1)$ then
  \[
    \mu(| \phi| \one_{\left\{ | \phi | \le \epsilon b_n \right\}}) \sim
    \frac{\alpha}{1 - \alpha} \epsilon b_n \mu(| \phi | > \epsilon b_n)
  \]

\item if $\alpha \in (1,2)$ then
  \[
    \lim_{n \to \infty} \frac{n}{b_n} \mu(|\phi |\one_{\{  \phi  > \epsilon
      b_n \}}) = \epsilon^{1 - \alpha} (2p -1) \alpha / (\alpha - 1)
    \]
\end{enumerate}
\end{prop}

\section{The deterministic billiard model}\label{sec:billiard-model}

We consider the modified Machta~\cite{Machta_83} billiard considered by
Jung, \Pene{} and Zhang \cite{Jung_Pene_Zhang2020}, which is a billiard
with flat cusp points. We recall the notation
of~\cite{Jung_Pene_Zhang2020}.

The billiard table $Q$ is a bounded region of $\R^2$, the boundary of which
consists of $q\ge 3$ dispersing $C^3$ curves $\Gamma_i$, such that
$\partial Q=\cup_{i=1}^q \Gamma_i$. The flat cusps are given by the
intersection of the curves $P_i=\Gamma_i \cap \Gamma_{i+1}$. In a
neighborhood of $P_i$ the boundary of $Q$ may be parametrized in an
Euclidean coordinate system $(x_i,y_i)$, for $x_i\in [0,\epsilon]$ with
some small $\epsilon>0$, by two curves $(x_i, y_{i,\pm}(x_i))$:
\[
  y_{i,\pm}=\pm C_{i,\pm} \frac{x_i^{\beta_i}}{\beta_i}+O(x_i^{2\beta_i-1})
  \text{\qquad and \qquad}
\frac{d y_{i,\pm}}{d x}=\pm C_{i,\pm}
  {x_i^{\beta_i-1}}+O(x_i^{2\beta_i-2})
\]
We assume  $\beta_i \ge 2$ for all $P_i$. We assume $C_{i,+}$, $C_{i,-}\ge 0$ are not both
zero, and that the unique tangent orbit out of $P_i$ hits $\partial Q$ at a
point that is not another cusp.
\begin{figure}[h]
\begin{center}
\begin{tikzpicture}
    \draw[black, ultra thick] (-4,0) .. controls (-1,0.1) and (0,1) .. (1,2);
    \draw[black, ultra thick] (-4,0) .. controls (-1,-0.1) and (0,-1) .. (1,-2);
    \draw[black, ultra thick] (1,2) .. controls (0.05,0.4) and (0.05,-0.4) .. (1,-2);
    \filldraw[black] (-4,0) circle (1.5pt) node[anchor=east]{$P_{1}$};
    \filldraw[black] (1,2) circle (1.5pt) node[anchor=south]{$P_{2}$};
    \filldraw[black] (1,-2) circle (1.5pt) node[anchor=north]{$P_{3}$};

    \draw[
    black,
    postaction={decorate},
    decoration={
        markings,
        mark=at position 0.1 with {\arrow{>}},
        mark=at position 1 with {\arrow{>}}
    }
]
(0.3, 0.2)
-- (-1.75, -0.27)
-- (-2.25, 0.16) 
-- (-2.50, -0.11) 
-- (-2.65, 0.08) 
-- (-2.66, -0.08) 
-- (-2.45, 0.11) 
-- (-2.10, -0.18)
-- (-0.9, 0.4);

\end{tikzpicture} 
\end{center}
\caption{Machta billiard with 3 cusps}
\end{figure}
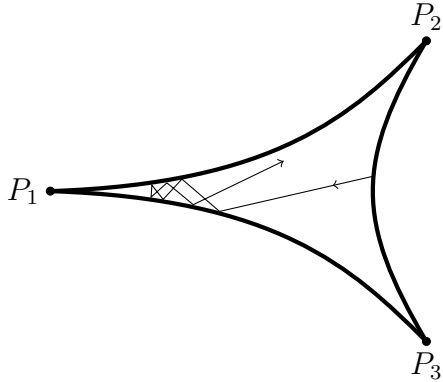

Let $J$ be the set of cusps.
Let
\[
  \betamax=\max \{\beta_i, i=1,\dots,q\}
\]
and assume $\betamax>2$. Let $\Jmax$ be the set of cusps with
$\beta_i=\betamax$.

Define\footnote{In \cite{Jung_Pene_Zhang2020} the inverse of $\gamma$ is
  denoted $\alpha$, so the limits there are $\alpha$-stable laws for H\"older observables. We use
  $\alpha$ as the stable index arising from the heavy tailed observable
   to allow easy comparison with~\cite{CNT2025}.}

   \[
   \gamma_i:=\frac{\beta_i-1}{\beta_i}
   \]
   and 
  \[
  \gamma:=\frac{\betamax-1}{\betamax}\in ({1}/{2},1)
\]
Note that $\frac{1}{\gamma}\in (1,2)$ will be the stable index for the
stable law arising from $\phi_1$.

We use the usual coordinate system of a billiard table $(r,\theta)$ where
$\theta\in [0,\pi]$ and $r\in [0,|\partial Q|]\equiv \R/|\partial Q|\Z$. We
write $\tilde{Q}$ for the phase space
$\tilde{Q}=\{(r,\theta): r\in \R/|\partial Q|\Z, \theta \in [0,\pi]\}$. In
this coordinate system we take 
\[
  d(x,x')=|r-r'|+|\theta-\theta'| \qquad \text{for $x=(r,\theta)$,
    $x'=(r', \theta')$.}
\]
 Since the space is compact our results also hold under the Euclidean
 metric.

We denote by $T:\tilde{Q}\to \tilde{Q}$ the usual billiard collision map,
angle of incidence equals angle of reflection, which preserves the measure
$d\mu=\frac{1}{2|\partial Q|}\sin\theta \; d r\; d\theta$,
see~\cite{Zhang2017,Jung_Zhang2018,Jung_Pene_Zhang2020}.

Let $\psi(r,\theta)$ be a H\"older or piecewise H\"older observable on
$\tilde{Q}$.

Denote by $r_i$ the $r$-coordinate of $P_i$ along $\partial Q$; at each
cusp point  $P_i$ define
\[
  \overline{\psi}_{i,\pm} (\theta)=\lim_{r \to r_i\pm} \psi(r,\theta)
\]
where the limits are taken along the upper and lower branches of the cusp.
For each  cusp in $J$ define the quantity
\begin{equation}\label{trap_i}
  I_{\psi,i}=\frac{1}{4} \int_{0}^{\pi} \left(\overline{\psi}_{i,+}
    (\theta) +\overline{\psi}_{i,-} (\theta)\right)\sin^{\gamma_i} \theta
  d\theta
\end{equation}
and
\[
I_i=\frac{1}{2}\int_{0}^{\pi}\sin^{\gamma_i} \theta d\theta
\]
The quantity $I_{\psi,i}$ roughly determines whether the trapping effect of the cusp $P_i$   is felt by the observable.

Note that at  each cusp  $P_i$ in $\Jmax$ the quantity
\begin{equation}
  I_{\psi,i}=\frac{1}{4} \int_{0}^{\pi} \left(\overline{\psi}_{i,+}
    (\theta) +\overline{\psi}_{i,-} (\theta)\right)\sin^{\gamma} \theta
  d\theta
\end{equation}
has the same density $\sin^{\gamma} \theta$.

\begin{rmk}\label{rmk:stronger-condition}
  A key estimate of Jung and Zhang is~\cite[Lemma 4.4 and Proof of Theorem
  4.1]{Jung_Zhang2018} where they show that if $I_{\psi,i}=0$ for
  all $i\in \Jmax$ then the Birkhoff sum
  $\sum_{j=1}^n (\psi \circ T^j-\mu(\psi))$ has variance which
  grows slower than $n^{2\gamma-\epsilon}$ for some $\epsilon>0$, hence
  $n^{-\gamma}\sum_{j=}^n (\psi\circ T^j-\mu(\psi))\to 0$ in
  distribution.

  This is roughly analogous  to
  the setting of the intermittent map, where if $\phi$ is a mean zero
  observable, under some H\"older regularity conditions, then
  $E(x):=\phi(x)-\phi(0)$ satisfies a CLT and hence Birkhoff sums of $E(x)$
  converge in distribution to zero under any scaling
  $n^{-\frac{1}{\alpha}}$, $0<\alpha<2$.

 \end{rmk}
 
\subsection{The reduction  to an observable on $M$}\label{contribution}

Jung, P\`ene and Zhang decompose $M$ further into a partition
$M=M_0\cup_{i=1}^m M_i$ where $M_i\subset M$ is the set 
 for which the orbits encounter the cusp $P_i$ first after leaving $M$, while $M_0$ is the set of 
 orbits which return to $M$ without encountering a cusp.
 
 Let $R$ be the return time function of $M$ to $M$ and $R_i=R \one_{M_i}$. Let $\mu_M$ be the conditional measure on $M$, that is $\mu_M (A) =\frac{\mu (A)}{\mu (M)}$ for $A\subset M$.
 
 Given a H\"older or piecewise H\"older function $\psi$ on $M$, Jung, P\`ene and Zhang~\cite{Jung_Pene_Zhang2020}
 show that the scaled Birkhoff sums of  the induced function of $\psi$ on $M$ are determined by those of the observable
 \[
 \tilde{R} (x)=\sum_{i\in J}\frac{I_{\psi,i}}{I_i} (R_i(x)-\mu_M (R_i))
 \]
 Note that if $\psi$ is constant, i.e. $\psi(x)=C$, then 
 \begin{equation}\label{constant}
 \sum_{i\in J}\frac{I_{\psi,i}}{I_i} (R_i(x)-\mu_M (r_i))=C(R(x)-\mu_M (R))
 \end{equation}

 Under the assumption that $\mu(\psi)=0$ and there is some cusp $P_i$,
 $P_i \in \Jmax$ with $I_{\psi,i}\not=0$, \cite[Theorem
 2.1]{Jung_Pene_Zhang2020} shows that
\[
  \frac{1}{n^{\gamma} } \sum_{j=0}^{n-1} \psi \circ T^j \dto
  \sum_{P_i\in \Jmax, I_{\psi,i}\not =0} X^{(i)}_{{\gamma}^{-1}}
\]
where $\sum_{P_i\in \Jmax, I_{\psi,i}\not =0} X^{(i)}_{1/\gamma}$ is a sum of independent
stable random variables of index $\gamma^{-1}$, so it is also a stable
random variable of index $\gamma^{-1}$. The precise parameters of each
$X^{(i)}_{1/\gamma}$ are described in~\cite[Theorem
2.1]{Jung_Pene_Zhang2020}.

In~\cite[Theorem
 2.1]{Jung_Pene_Zhang2020} the components $X^{(i)}_{{\gamma}^{-1}}$
 arise by showing the function $\sum_{i\in J} \frac{I_{\psi,i}}{I_i} [R_i (x)-\mu_M (R_i)]$
 satisfies a  stable law with independent stable variables 
 given by the functions $ \frac{I_{\psi,i}}{I_i}[R_i (x)-\mu_M (R_i)]$. The components with $\gamma_i <\gamma$
 converge to zero in distribution under the scaling $n^{-\gamma}$. This induced stable law is then lifted to give the 
 stable law on $\tilde{Q}$.

 Furthermore,~\cite[Theorem 2.2]{Jung_Pene_Zhang2020} (see also~\cite{Melbourne_Varandas2020})
shows that, under the
additional assumption that for each $P_i \in \Jmax$, ${\psi}|U_i \ge 0$ or
$\psi| U_i \le 0$ on a neighborhood of $P_i$
\[
  W_n (t):=\frac{1}{n^{\gamma}} \sum_{j=0}^{[nt]-1} \psi\circ T^j
\]
converges in the Skorokhod $M_1$-topology to a $\gamma^{-1}$-stable L\'evy
process.

\section{Main Results}\label{sec:main-results}

\subsection{Stable laws for billiards with
  cusps}\label{sec:billiards-with-cusps}

Consider $\phi (x)=d(x,x_0)^{-\frac{2}{\alpha}}$, with $x_0$ not a cusp
point. If $1<\alpha<2$, then $\phi$ is integrable and we define
$\tilde{\phi}=\phi-\mu (\phi)$. Following the notations in
Section~\ref{sec.background}, we define at each cusp point $P_i$
\[
  \tilde{\phi}_{i,\pm} (\theta)=\lim_{r \to r_i\pm} \tilde{\phi}(r,\theta)
\]
and 
\begin{equation}
  I_{\tilde{\phi},i}=\frac{1}{4} \int_{0}^{\pi} \left(\tilde{\phi}_{i,+}
    (\theta) +\tilde{\phi}_{i,-} (\theta)\right)\sin^{\gamma} \theta
  d\theta
\end{equation}

We will make the following assumptions on the point $x_0$, satisfied by $\mu$ a.e. point.
\begin{itemize}
\item[ (1)]  $x_0$ is not a cusp point;
\item[ (2)]  $x_0$ satisfies the no-short return conditions of~\cite[Section 4.1.2]{GuptaHollandNicol2011};
\item[ (3)] $x_0$ satisfies the assumptions of~\cite{Pene-Saussol2020}.
\end{itemize}

\begin{thm}\label{main1}
  Suppose $(T,\widetilde{Q},\mu)$ is the boundary collision map on dispersing
  billiards with cusps $P_i$, $i=1,\dots,q$, as in
  Section~\ref{sec:billiard-model}. Let
  $\phi(x)=d(x,x_{0})^{-\frac{2}{\alpha}}$ with $0<\alpha<2$,
  $\alpha \not =1$, where
  $x_{0}$ is not a cusp point. If $\alpha \in (1,2)$ define
  $c:= \mu (\phi)$, otherwise $c:=0$. Let
  $\gamma := \frac{\betamax -1}{\betamax}\in(1/2,1)$, with $\betamax>2$.

  Then for $\mu$ a.e. $x_0$:
  \begin{itemize}

  \item [(a)] If $\frac{1}{\alpha} > \gamma$ then
    \[
      \frac{1}{n^{1/\alpha} }\sum_{j=1}^{n}[\phi \circ T^j -c] \dto
      X_\alpha
    \]
    where $X_\alpha$ has a stable distribution of index $\alpha$;
			
  \item [(b)] If $\frac{1}{\alpha} < \gamma$ (this implies $\alpha>1$) and
    $I_{\tilde{\phi},i} \not =0$ for some $i\in \Jmax$, then
    \[
      \frac{1}{n^{\gamma}} \sum_{j=1}^{n}( \phi \circ T^j -
      \mu(\phi)) \dto X_{\gamma^{-1}}
    \] 
    where $X_{1/\gamma}$ has a stable distribution of index $1/\gamma$.

\end{itemize} 
\end{thm}

\begin{rmk}
The distribution of $X_{1/\gamma}$ is given by lifting the stable law of the induced function

\[
\sum_{i\in \Jmax, I_{ \tilde{\phi}, i} \not =0} \frac{I_{\tilde{\phi},i}}{I_i} [R_i (x)-\mu_M (R_i)]
\]

\end{rmk}

The next theorem is in the setting of a mean-zero observable $\phi$, not necessarily of 
form $\phi(x) =d(x,x_0)^{-\frac{2}{\alpha}}$, with  support bounded away from each cusp. Roughly it states
that it is possible to  lift an induced stable law  on a slowly-mixing tower if the induced observable 
 and the 
induced dynamics satisfy the assumptions of Theorem~\ref{distribution}.

\begin{thm}\label{main2}
  For each $i=1,...,q$ let $O_i$ be an open set
  containing $P_i$ and $U=\cup_{i=1}^q O_i$. 
  Suppose $ \phi \equiv 0$ on $U$, $\int \phi d\mu=0$,  and the induced version of $\phi$ on  a subset $M\subset \tilde{Q}$ 
  satisfies a stable law of index $\alpha$, where $\gamma>\frac{1}{\alpha}$.

  Then the stable law lifts and
    \[
      \frac{1}{n^{1/\alpha}} \sum_{j=1}^{n}( \phi \circ T^j -
      \mu (\phi) )\dto X_{\alpha}
    \]
    where $X_{\alpha}$ has a stable distribution of index $\alpha$.
\end{thm}

\subsection{Stable laws for intermittent maps}\label{sec:intermittent-maps}

As a consequence of Theorem~\ref{distribution} we are able to extend the parameter range
for which we have convergence to stable law of a scaled Birkhoff sum of an 
observable of stable index $\alpha$  on a  class of intermittent type maps. The 
maps 
$T_{\gamma} : [0,1] \rightarrow [0,1]$  are defined in~\cite{LSV99}, and given by
\begin{equation}\label{IM}
  T_{\gamma} (x) := 
  \left\{
    \begin{array}{ll}
      (2^{\gamma}x^{\gamma} +1)x& \mbox{if $0 \leq x<\frac{1}{2}$};\\
      2x-1 & \mbox{if $\frac{1}{2} \le x \le 1$}.
    \end{array}
  \right.
\end{equation}
For $\gamma \in [0,1)$, there is a unique absolutely continuous (ergodic)
invariant probability measure $\mu_{\gamma}$ with density $h_{\gamma}$
bounded away from zero and satisfying $h_{\gamma}(x) \sim Cx^{-\gamma}$ for
$x$ near zero. The observable we consider is again
$\phi(x)=d(x,x_0)^{-1/\alpha}$, $x_0\in [0,1]$.

\subsubsection{Stable laws for the observable $\phi(x)=d(x, x_0)^{-\frac{1}{\alpha}}$,
  $x_0\in (0,1]$ on an LSV map.}\ \null

An application of Theorem~\ref{distribution} allows us to improve the
results of~\cite{CNT2025} to include the  parameter range 

$\alpha\in (0,2)$,  $\gamma \in [0,1)$, $\alpha \not =1$, $\gamma\not= \frac{1}{\alpha}$ extending from the
restricted range
$\frac{1}{\gamma}< \alpha <1-\frac{1}{\gamma}+\frac{1}{\gamma^2}$
of~\cite[Theorem 8.4 case (iii)]{CNT2025}. The new result is part (c)
below.
  
\begin{prop}\label{thm:intermittent}
  Let $(T_{\gamma}, X, \mu)$ be an LSV map of the unit interval with
  $0\le \gamma <1$ and $\alpha \in (0,1)\cup(1,2)$. Suppose
  $\phi (x) =d(x,x_0)^{-\frac{1}{\alpha}}$ where $x_0\in (0,1]$. If
  $\alpha\in (1,2)$ (so $\phi$ is integrable) then let
  $c_n=\mu_\gamma(\phi)=\int d(x,x_0)^{-\frac{1}{\alpha}}d\mu_{\gamma}$,
  for $\alpha\in (0,1)$ let $c_n=0$.
  \begin{itemize}

  \item [(a)] If $\frac{1}{\alpha}> \gamma$ then
    \[
      \frac{1}{n^{1/\alpha} }\sum_{j=1}^{n}[\phi \circ T^j -c_n] \to_{d} X_{\alpha}
    \]
    where $X_\alpha $
    has a stable distribution of index $\alpha$.

  \item [(b)] If $\frac{1}{\alpha} < \gamma$ (so $\gamma>1/2$) and
    $\phi (0)-\mu (\phi) \not =0$ then
    \[
      \frac{1}{n^{\gamma}} \sum_{j=1}^{n}[\phi \circ T^j - \mu (\phi)]
      \to_{d} X_{\frac{1}{\gamma}}
    \]
    where $X_{\frac{1}{\gamma}}$ has a stable distribution of index $\frac{1}{\gamma}$.

  \item [(c)] If $\frac{1}{\alpha} < \gamma $ and $\phi(0)-\mu (\phi)=0$
    then
    \[
      \frac{1}{n^{1/\alpha}}  \sum_{j=1}^{n} [\phi \circ T^j - \mu (\phi)]
      \to_{d} X_\alpha
    \]
    where $X_\alpha$ has a stable distribution with index $\alpha$.
    
  \end{itemize}
\end{prop}

\subsubsection{Stable laws for $\phi(x)=x^{-\frac{1}{\alpha}}$, Noah and Joseph effect combined.}\ \null
  
We extend the results of~\cite{Gouezel_Intermittent}, for the observable
$\phi(x)=x^{-\frac{1}{\alpha}}$, maximized at the indifferent fixed
point.

Gou\"ezel~\cite[Item 3 after the proof of Theorem
1.3]{Gouezel_Intermittent} has results for
$\phi(x)=x^{-\frac{1}{\alpha}} +\kappa$, $1<\alpha<2$,
$\frac{1}{\alpha}+\gamma <1$. In this range $\phi$ is integrable, and
$\kappa$ is taken so that $\int \phi d {\mu_{\gamma}}=0$. Using operator
renewal theory, Gou\"ezel shows that: if
$\frac{1}{\alpha}+\gamma<\frac{1}{2}$ then $\phi$ satisfies a central limit
theorem; if $\frac{1}{\alpha}+\gamma=\frac{1}{2}$ then $\phi$ satisfies a
central limit theorem with scaling $\frac{1}{\sqrt{n\ln(n)}}$; and finally,
if $\frac{1}{2}< \frac{1}{\alpha}+\gamma<1$ then $\phi$ satisfies a stable
law with scaling $n^{-(\frac{1}{\alpha}+\gamma)}$. He remarks that ``even
when $\gamma<\frac{1}{2}$ we can have convergence to a stable law when
$\phi$ is unbounded.''
  
We extend these results to the parameter range
$\{ (\alpha,\gamma): 0<\gamma<1, 0<\alpha<2, \alpha\not =1\}$ in the
following theorem:
  
\begin{thm}\label{cusp_intermittent}
  Let $(T_{\gamma},[0,1],\mu_{\gamma})$ be the LSV map~\eqref{IM} with
  $0<\gamma <1$. Let $\phi_0(x)=x^{-\frac{1}{\alpha}}$ and consider
  $\phi(x)=x^{-\frac{1}{\alpha}}- \mu_\gamma(\phi_0)$ if $\phi_0$ is
  $\mu_\gamma$-integrable, and $\phi(x)=x^{-\frac{1}{\alpha}}$ otherwise.
  Suppose $0<\alpha <2$, $\alpha\not =1$, $\gamma+\frac{1}{\alpha}\not =1$ and
  $\frac{1}{2}<\frac{1}{\alpha}+\gamma$.
  
  Then 
  \[
    \frac{1}{n^{\frac{1}{\alpha}+\gamma}}\sum_{j=0}^n \phi\circ
    T_{\gamma}^j\dto X_{\left(\frac{1}{\alpha}+\gamma\right)^{-1}}
 \]
 where $X_{\left(\frac{1}{\alpha}+\gamma\right)^{-1}}$ has stable index
 $\left(\frac{1}{\alpha}+\gamma\right)^{-1}$.
\end{thm}

\begin{rmk}
  Note that the stable index $\left(\frac{1}{\alpha}+\gamma\right)^{-1}$ is
  not due to the unbounded density of the indifferent fixed point of form
  $h(x)\sim x^{-\gamma}$. Indeed,
  $\mu_\gamma(x: x^{-\frac{1}{\alpha}}>t)\sim \int_0^{t^{-\alpha}}
  x^{-\gamma}dx\sim t^{-\alpha(1-\gamma)}$ with stable index
  $\alpha(1-\gamma)$ and not $\left(\frac{1}{\alpha}+\gamma\right)^{-1}$.
  Note that
  $\left(\frac{1}{\alpha}+\gamma\right)^{-1} > {\alpha(1-\gamma)}$ if and
  only if $\frac{1}{\alpha}+\gamma>1$. By~\cite[Theorem 2.3]{Carney_Nicol}
  if $\beta>\frac{1}{\alpha}+\gamma>1$ then
\[
\frac{1}{n^{\beta}} \sum_{j=0}^{n-1} \phi \circ T^j_{\gamma} (x)  \rightarrow 0
\]
for $\mu_{\gamma}$ a.e. $x$, so
$\mu_\gamma(x: x^{-\frac{1}{\alpha}}>t)\sim t^{-\alpha(1-\gamma)}$ does not
give the correct scaling in the case $\frac{1}{\alpha}+\gamma>1$.

\end{rmk}

Here is a brief outline of the paper. We isolate a subset
$M\subset \tilde{Q}$, away from the cusps but containing the point $x_0$
where $\phi$ is unbounded.  The return time function $R$ to $M$ has polynomial tails,
$\mu (R>n)\sim n^{-\frac{1}{\gamma}}$. The return time map $F:=T^R: M\to M$ has the structure of a Young Tower with
exponential tails. We decompose the observable $\phi:\tilde{Q}\to\R$,
 as
\[
  \text{$\phi=\phi_1+\phi_2$ where $\phi_1=\phi \one_{M^c}$ and
    $\phi_2=\phi \one_{M}$.}
\]
The point of this decomposition is to isolate the two effects: heavy tails and slow mixing.

For $\phi_1:\tilde{Q}\to \R$, Jung and Zhang~\cite{Jung_Zhang2018} (see also Jung, \Pene~and Zhang~\cite{Jung_Pene_Zhang2020}, Melbourne and Varandas \cite{Melbourne_Varandas2020})
prove a stable law of index $1/\gamma$, provided not all quantities $I_{\tilde{\phi},i}$, $i\in \Jmax$, vanish.   We show $\phi_1$ induces as a locally H\"older continuous
function, $\sum_{j=0}^{R-1} \phi_1\circ T^j$,  on  the partition elements of $M$  and so fits into the setting in which the results of~\cite{Jung_Zhang2018} hold.   Since the scaled Birkhoff sums of the induced function converge in the $J_1$ topology
to  a L\'evy process of index $\frac{1}{\gamma}$ in the $J_1$ topology the induced limit law may be lifted to that of the original observable
$\phi_1$ by the continuous mapping theorem.

Next  we induce  $\phi_2:M\to \R$ to the observable $\Phi_2=\sum_{j=0}^{R-1} \phi_2\circ T^j$.
We show that   under the
induced map $F:M\to M$,  a suitably scaled Birkhoff sum $n^{-\frac{1}{\alpha}} \sum_{j=1}^n(\Phi_2\circ F^j-c)$
converges in distribution to a  stable law of index $\alpha$, see
Section~\ref{sec:stable-laws}. In the case $1<\alpha<2$ this necessitates proving a form of  ``vanishing small values'' condition 
in Theorem~\ref{thm:mixing-condition_Young-Tower}.

Consider now the cases in Theorem~\ref{main1}:

\begin{itemize}
\item [(a)] $\frac{1}{\gamma}>\alpha$; we lift the stable law for $F$ on
  $M$ to a stable law for the billiard map $T:\tilde{Q}\to\tilde{Q}$ using
  a result of Gou\"ezel~\cite{Gouezel_Doubling}, see
  Proposition~\ref{gouezel-lifting}. Combining the two stable laws, the
  one that dominates is that of index $\alpha$ (a use of Slutsky's theorem).

\item [(b)] since $\frac{1}{\gamma}<\alpha$, the $\frac{1}{\gamma}$-stable law of
  $\phi_1$ dominates.
  
  \end{itemize}

In Theorem~\ref{main2} we use Theorem~\ref{distribution} to lift an induced $\alpha$ stable limit law
to the original observable, as standard lifting techniques (the Continuous Mapping Theorem or conditions
such as those of Proposition~\ref{gouezel-lifting} ) do not apply.

We use Theorem~\ref{distribution} to extend the results of \cite{CNT2025} for heavy-tailed observables
on LSV maps $(T_{\gamma},[0,1],\mu_{\gamma})$ in Proposition~\ref{thm:intermittent} to the entire parameter region $0<\gamma <1$, $0<\alpha<2$ ($\alpha \not =1$). In Theorem~\ref{cusp_intermittent} we consider a heavy tailed observable maximized at the indifferent fixed point: $\phi (x)=x^{-\frac{1}{\alpha}}-\mu_{\gamma}(x^{-\frac{1}{\alpha}})$
if $ \mu_{\gamma}(x^{-\frac{1}{\alpha}})<\infty$ (equivalently, if
$\frac{1}{\alpha}+\gamma<1$), or $\phi (x)=x^{-\frac{1}{\alpha}}$ if
$x^{-\frac{1}{\alpha}}$ is not $\mu_{\gamma}$-integrable. We extend the results of Gou\"ezel to obtain a stable law
of index $\left(\frac{1}{\alpha}+\gamma\right)^{-1}$ for all $0<\gamma <1$, $0<\alpha<2$ ($\alpha \not =1$).

\section{Lifting from the induced system}\label{sec:lifting-continuous-mapping}

If our induced system converges in the $J_1$ or $M_1$ topology to a
continuous time stable L\'evy process lifting to a stable law for the original observable may be done
with the Continuous Mapping Theorem. We use Theorem~\ref{distribution} to lift
if we do not have, or cannot prove, convergence in  the $M_1$ or $J_1$ topologies (due to the process switching
signs or other reasons). 

\subsection{Lifting via functional convergence and  the Continuous Mapping Theorem}\label{subsec:lifting-continuous-mapping}

Suppose $\alpha\in (0,2)$, $\alpha \not =1$.
Suppose the continuous time version of our induced system $(\Phi_2, M,\mu_M)$ converges in Skorokhod's
$J_1$ (or $M_1$) topology, so that for $t\in [0,2]$
\[
\frac{1}{n^{\frac{1}{\alpha}} }\sum_{j=0}^{\lfloor nt\rfloor} [\Phi_2\circ F^j- \mu_M (\Phi_2)]
\]
converges to a continuous time  stable L\'evy process of index $\alpha$ in the
$J_1$ ($M_1$) topology. If $\alpha\in (0,1)$ we do not need to center.

By the ergodic theorem
\[
\frac{R_n(x)}{n}\to \frac{1}{\mu(M)}:=\bar{R}
\] 
for $\mu$ a.e. $x\in M$.


We define the following sequence of random variables on $M$, indexed by
$t\in[0,2]$
\[
  z_n (t) (x)= \left\{
    \begin{array}{ll}
      \frac{R_n(x)}{\bar{R}n}t& \mbox{if $t \in [0,1]$};\\
      2+ (2-\frac{R_n(x)}{\bar{R}n})(t-2)& \mbox{if $t\in [1,2]$}.
    \end{array}
  \right.
\]

The sequence $z_n(t)$ converges in $C[0,2]$ to the identity function $z(t)=t$. This convergence
holds in the uniform topology since
the sequence converges $\mu$ a.e. and hence weakly. The composition map
\[
  (f,g)\to f\circ g
\]
is continuous from $D([0,2])\times C([0,2])$ to $D([0,2])$ in the $J_1$ ($M_1$)
topology. 

By the Continuous Mapping Theorem
\[
  \left(\frac{R_n(x)}{\bar{R}}\right)^{-\frac{1}{\alpha}}
  \sum_{j=0}^{\lfloor \frac{R_n(x)}{\bar{R}}t\rfloor}( \phi_2\circ
  F^j-\mu(\phi_2)) 
\]
converges to a continuous time stable process of index $\alpha$ in  the $J_1$ ($M_1$) topology.

Thus for $t=1$
\[
\lim_{n\to \infty}   n^{-\frac{1}{\alpha}} \sum_{j=0}^{n}( \phi_2\circ T^j -\mu (\phi_2))
\]

\[
=\lim_{n\to \infty}   \left(\frac{R_n(x)}{\bar{R}}\right)^{-\frac{1}{\alpha}} \sum_{j=0}^{[ \frac{R_n(x)}{\bar{R}}t]}( \phi_2-\mu(\phi_2))
\]
converges in distribution to a stable law of index $\alpha$.
    
By results of~\cite[Theorem 2.2]{Jung_Pene_Zhang2020}, under additional
assumptions
\[
  n^{-\gamma} \sum_{j=0}^{[nt]}(\phi_1\circ T^j -\mu (\phi_1)
\]
satisfies a functional stable law, and converges in the $M_1$ topology
to a  stable L\'evy process of index~$\frac{1}{\gamma}$.

In our setting, the point process associated to the induced version of $\phi_2$ on M,
denoted $\Phi_2$, converges to a compound Poisson Point Process as a consequence of 
results of \Pene{} and Saussol~\cite{Pene-Saussol2020}. If $0<\alpha<1$
then by a result of Davis and Hsing~\cite{DH95}, $\Phi_2$ satisfies a
stable law under suitable scaling but if $\alpha \in (1,2)$ then to obtain
a stable law we must also verify a ``vanishing small values''
condition~\cite[Theorem 3.1]{DH95}. This we do for generic $x_0$ in
Theorem~\ref{thm:mixing-condition_Young-Tower}. If we had convergence of
the continuous time version of $\Phi_2$ in the $M_1$  topology it would allow us to lift, via the Continuous
Mapping Theorem, to convergence of $\phi_2$ to a stable law on $\tilde{Q}$
but not automatically  to a stable L\'evy process. 
Conditions under which a compound Poisson Point process converges to a
stable L\'evy process in the $M_1$ topology have been given by Basrak et
al.~\cite[Theorem 3.4]{Basrak}. Once again in the case $\alpha \in (0,1)$
it is automatic to obtain convergence to a functional stable L\'evy process
in the $M_1$ topology~\cite[Theorem 3.4]{Basrak}. But in the case
$\alpha \in (1,2)$ to obtain convergence it is necessary, to invoke
~\cite[Theorem 3.4]{Basrak}, to verify a ``vanishing small values''
condition~\cite[Condition 3.3]{Basrak} which is stricter than that
of~\cite[Theorem 3.1]{DH95}. We are not able to verify~\cite[Condition
3.3]{Basrak} in our setting. Thus in the case $0<\alpha<1$ we have convergence 
of the associated continuous time process to a stable L\'evy process
in the $M_1$ topology and via the Continuous Mapping Theorem we
are able to lift to a stable law of the same index on the original system.
In the case $1<\alpha <2$ we are not able to show convergence in the $M_1$ topology.
Instead we verify a ``vanishing small values'' condition in Theorem~\ref{thm:mixing-condition_Young-Tower} and prove convergence
to a stable law.  Then we lift the stable law from the induced system to the original in Section~\ref{sec:proof-main}. Thus we restrict our results to a
discussion of stable law convergence.

\subsection{A maximal negligibility result for exponentially mixing systems.} 

 Theorem~\ref{distribution} is useful  to lift
if we do not have, or cannot prove, convergence in  the $M_1$ or $J_1$ topologies. The proof is given in the Appendix.

\begin{thm}\label{distribution}
  Let $(F,X,\mu)$ be an ergodic dynamical system with $X$ a Riemannian
  manifold of dimension $\dimX$ and $\mu$ an invariant probability measure.
  Let $B_1$ and $B_2$ be Banach spaces. Suppose that $(F,X,\mu)$ has
  exponential decay of correlations for $B_1$ versus $B_2$ of the form:
  there are $C>0$ and $0<\theta<1$ such that
  \[
    |\int f g\circ F^n d \mu - \mu (f)  \mu ( g) |\le C\theta^n
    \|f\|_{B_1} \|g\|_{B_2}, \qquad n \ge 0
  \]

 Let $\phi(x)$ be an observable on $(X,\mu)$ with $\mu(|\phi|>t)\le Ct^{-\alpha}$ for
 $\alpha \in (1, 2)$.
 
  Let $\epsilon>0$. Truncate $\phi$ by defining
 $\psi_n:=\phi \one_{A_n} $ for measurable sets $A_n$ with $\|\phi\one_{A_n^c}\|_{L^1(\mu)}\le Cn^{\frac{1}{\alpha} -1}$, $\|\psi_n\|_{\infty}\le Cn^{\frac{1}{\alpha}+\epsilon}$,
 $\mu (A_n^c)\le Cn^{-1-\alpha\epsilon}$. Let
 $\bar{\psi}_n:=\psi_n-\bE_\mu (\psi_n )$.

 Assume there exist approximations $f_n \in B_{i}$, $i=1,2$, to
 $\bar{\psi}_n$ with the properties:
 \begin{itemize}
 \item[(a)] $\|\bar{\psi}_n-f_n \|_{L^1(\mu)}< C n^{-4-\frac{1}{\alpha}}$
 \item[(b)] $\|f_n\|_{B_i}<C n^{\kappa}$ ($i=1,2$) \mbox{for some $\kappa>0$}
 \item[(c)] $\|f_n\|_{\infty}\le C \|\bar{\psi}_n\|_{\infty}$. 
 \item[(d)] $\mu(f_n)=0$
 \end{itemize}
 where $C>0$.
 
 (The approximations $f_n$ depend on $\epsilon$ but for ease of notation we
 do not indicate this dependence).
 Then 
  \begin{equation}\label{eq:conv-in-distribution}
    \frac{1}{n^{1/\alpha+\epsilon}}
    \max_{j\le n}\left|\sum_{i=1}^{j}  \left[{\phi}\circ F^i - \mu (\phi)\right]\right|
    \dto 0
      \end{equation}
\end{thm}

\begin{rmk}\label{rmk:Banach-pairs}
  The Banach spaces, in applications, are often the pairs: $B_1=\Lip(X)$,
  $B_2=\Lip(X)$ (invertible case); $B_1=BV(X)$, $B_2=L^1(X)$ or
  $B_1=\Lip(X)$, $B_2=L^1(X)$ (non-invertible case).

  \end{rmk}

For observables similar to those  we consider in this paper, we have  the following
corollary.

\begin{cor}\label{distribution-corollary}
  Consider the dynamical system $(F,X,\mu)$ described in
  Theorem~\ref{distribution}, with $X$ of dimension $D$.

  Assume the observable $\phi:X\to \R$ of stable index $\alpha\in (1,2)$ is
  given by
  \[
    \phi(x)=\sum_{i=1}^{p} C_i d(x,x_i)^{-\frac{D}{\alpha}}, \quad
    C_i\in (-\infty,\infty)
  \]
  and that $\mu$ is equivalent to Lebesgue in a neighborhood of each point
  $x_i \in X$.

  Then
  \begin{equation}\label{eq:conv-in-distribution_all-epsilon}
    \frac{1}{n^{1/\alpha+\epsilon}}
    \max_{j\le n}\left|\sum_{i=1}^{j} \left[{\phi}\circ F^i - \mu (\phi)\right]\right|
    \dto 0 \qquad \text{ for any $\epsilon >0$.}
  \end{equation}
\end{cor}

\section{The Tower construction}\label{sec:tower-construction}

A very useful inducing scheme has been given by~\cite{Markarian04,CZ2005},
by removing areas of weak hyperbolicity we induce on a subset such that the induced
map is uniformly hyperbolic and has good mixing properties. We then lift statistical limit laws
from the induced system to the original.

To be concrete, fix an integer
$K_0$ and define $M$ as the set
\[
  M=\{p\in \tilde{Q} : \mbox{ less than $K_0$ successive forward collisions
    of $p$ in any cusp} \}
\]
We consider the first return time $R: M\to \Z$ and the induced map
\[
F=T^R :M\to M
\]
The measure $\mu_M=\frac{1}{\mu(M)}\mu$ is preserved by $F$.

It is known that $(F,M,\mu_M)$ may be modeled by a Young Tower with
exponential tails~\cite{CZ2005}.

We consider an observable $\phi: \tilde{Q}\to \R$ of form

\[
  \phi (x)= d(x,x_0)^{-\frac{2}{\alpha}}
\]
where $d(x,x')=|r-r'|+|\theta-\theta'|$, $x=(r,\theta)$, $x'=(r',\theta')$.

As long as $x_0$ is not a cusp point  we may, by
appropriate choice of $M$, assume that    $x_0$ lies in the inducing set $M$. Later we will make
a  generic ``no short returns'' condition on $x_0$. Generic here means belonging to a set of full measure.

If $0<\alpha <1$ we decompose 
\[
  \phi= \phi_1 +\phi_2
\]
where 
\[
\phi_1=\phi_{| M^c}
\]
and 
\[
\phi_2= \phi_{|M}
\]
If $1<\alpha <2$ we decompose
\[
\phi- \mu (\phi)= (\phi_1 -\mu (\phi_1 ) +(\phi_2-\mu (\phi_2)
\]

From work of Jung et al.~\cite{Jung_Pene_Zhang2020} we know that
\[
  \frac{1}{n^{\gamma} } \sum_{j=1}^n( [\phi_1 \circ T^j -\mu (\phi_1)\dto
  \sum_{i \in \Jmax, I_{\tilde{\phi},i}\not =0} X^{(i)}_{\gamma^{-1}}
\]
provided not all the quantities $I_{\tilde{\phi},i}$, $i\in \Jmax$,
vanish.

The function $\phi_1-\mu (\phi)$   fits into the setting in which the results of~\cite{Jung_Zhang2018} hold.  The same 
analysis as in~\cite[Section 2.2]{Jung_Zhang2018} proves that $\phi_1$ induces as a locally H\"older continuous
function, $\sum_{j=0}^R \phi_1\circ T^j$,  on  the partition elements of $M$. In fact $\phi_1\equiv 0$ on $M$ and equals $\phi$ on $M^c$
so there is no modification needed to their proof.



As in~\cite{CNT2025} we will investigate the interplay between the
$\alpha$-stable law arising from the heavy-tailed observable $\phi_2$ and
the $1/\gamma$-stable law arising from $\phi_1$ due to the slow-mixing of
the billiard.

Induce $\phi_2$ on $M$. As $\phi_2=0$ on $M^c$, the induced version of
$\phi_2$,
\begin{equation*}
  \Phi_2:=\sum_{j=0}^{R(p)-1} \phi_2 \circ T^j:M\to \R
\end{equation*}
equals $\phi_2$, i.e. $\Phi_2=\phi|_M$.

Next we show that $\Phi_2$ satisfies a stable law (or convergence in $J_1$
topology) under $(F,M,\mu_M)$. This follows directly from the
work~\cite{Pene-Saussol2020} in the case $0<\alpha<1$. In the parameter
range $1<\alpha <2$ we still rely on~\cite{Pene-Saussol2020} but must also
verify a mixing condition akin to~\cite[Theorem 3.1]{DH95} in the setting
of a uniformly hyperbolic system with singularities modeled by
exponentially mixing Young Tower. This second case uses recurrence
estimates for generic points given in~\cite{GuptaHollandNicol2011}.

We then prove that we can lift the stable law (or $J_1$ convergence) from
$\Phi_2$ under $(F,M,\mu_M)$ to a corresponding stable law for $\phi_2$
under $(T,\tilde{Q},\mu)$ in case (a) of Theorem~\ref{main1} and in Theorem~\ref{main2}.

\section{Convergence of $\left(j/n,(\Phi_2 \circ F^{j-1}-c_n)/b_n\right)$ to a  Poisson process}

In the setting of a uniformly hyperbolic dynamical system with
singularities which is modeled by a Young Tower,
\Pene{} and Saussol have verified the assumptions of~\cite[Hypothesis
3.1]{Pene-Saussol2020}. The induced dynamics $(F,M,\mu_M)$ fits into this scheme,
satisfies the assumptions of~\cite[Hypothesis 3.1]{Pene-Saussol2020}. The results of ~\cite{Pene-Saussol2020}
concern hits to a shrinking ball $A_{\epsilon}$ which we take as sequence of balls  $U_n(x_0)$ of radius approximately
$\frac{1}{\delta\sqrt{n}}$, $0<\delta<1$, about a generic
point $x_0$
(in particular satisfying the short-returns estimate of
\cite[Section 4.1.2]{GuptaHollandNicol2011} used in the proof of
Theorem~\ref{thm:mixing-condition_Young-Tower}, so that $x_0$ is
aperiodic under $F$).
Our map $H_n$, playing the role of $H_{\epsilon}$ in~\cite[Section
2.1]{Pene-Saussol2020} will be the map
$\frac{1}{b_n} \Phi_2: U_n (x_0)\to [\delta^{\frac{2}{\alpha}},\infty)$.
Here the role of $V$ is taken by $[\delta^{\frac{2}{\alpha}},\infty)$. The
measure $m_{\epsilon}$ on $V$ in our setting is the measure $\Pi_{\alpha}$
with density $C\alpha x^{-\alpha-1}$.


Define the counting process $N_n$ on $[0,1]\times [\delta^{\frac{\alpha}{2}},\infty)$ by
\[
  N_n (J \times B)=\#  \{ j\le n:( j/n,(\Phi_2 \circ F^{j-1}-c_n)/b_n) )\in
    J\times B \}
\]
For each $x\in M$, $N_n (x) $ is an integer valued counting process
on $[0,1] \times [\delta^{\frac{\alpha}{2}},\infty)$. It counts the number of hits to 
$U_n (x_0)$ and also where they occur in $U_n (x_0)$, which determines the counting process  on 
$[\delta^{\frac{2}{\alpha}},\infty)$.

Under our assumption that $x_0$ is generic~\cite[Theorem 2.1]{Pene-Saussol2020} shows that 
$N_n$ converges to a Poisson process with intensity $m \times \Pi_{\alpha}$ on $[0,1] \times [\delta^{\frac{\alpha}{2}},\infty)$.

As $\delta$ was arbitrary this  implies that the related 
counting process on $\R\setminus\{0\}$ (i.e. we take $J=[0,1]$ in the 
first component) defined by, for $B\in \R\setminus\{0\}$,
\[
  N_n (B) =\# \{ j \in [1,n]: (\Phi_2 \circ F^{j-1}-c_n)/b_n \in B \}
\]
converges to a process which has the general form of~\cite[Corollary
2.4]{DH95}, namely
\begin{equation}\label{eq:counting-process-convergence}
  N_n \to N=\sum_{i=0}^{\infty} \delta_{P_i}
\end{equation}
where $\sum_{i=1}^{\infty} \delta_{P_i}$ is a Poisson process with
intensity measure $\Pi_{\alpha}$.

\section{Stable laws}\label{sec:stable-laws}

If $0<\alpha <1$ then, by~\cite[Theorem 3.1]{DH95}, the existence of a
Poisson limit $N$ given in \eqref{eq:counting-process-convergence} implies
that $\{ \Phi_2\circ F^j\}$ satisfies a stable law of index $\alpha$. In
order to establish a stable law in the regime $1<\alpha <2$ it suffices,
by~\cite{DH95}, to show in addition that
\begin{equation}\label{eq:D9bis}
  \lim_{\epsilon \to 0} \limsup_{n\to
    \infty}\frac{n}{b_n^2}\sum_{j=1}^{n} \max\left\{0, \int_M \Phi_n \cdot \Phi_n\circ F^{j}
    d\mu\right\}=0,
\end{equation}
where, for a fixed $\epsilon>0$, we denote
\begin{equation}\label{eq:phi_n}
  \phi_n :=\phi_2\cdot \one_{\left\{|\phi_2| \le \epsilon b_n\right\}}
  \text{ and } \Phi_n := \phi_n - \mu_M (\phi_n).
\end{equation}

In general it is difficult to establish Equation~\eqref{eq:D9bis} directly
from decay of correlations. 

We will use the fact that $F:M\to M$ can be modeled by a Young Tower and
make generic assumptions on the short returns of $x_0$, which has been
verified in our setup in~\cite[Section 4.1.2]{GuptaHollandNicol2011}.

\begin{rmk}
  The Young Tower analyzed for the intermittent map in~\cite{CNT2025}
  allowed us to obtain results for every $x_0\not =0$ (where $0$ was the
  indifferent fixed point) since it has a better topological structure than
  the Young Tower for hyperbolic billiards. In fact, in the intermittent
  map case a Tower was constructed in which the preimage of $x_0$ appears
  in precisely one column of the Tower and every point $x$ lies in the
  Tower. In the billiard case, the Tower contains $\mu$ a.e. $x$ and there
  may be infinitely many preimages of $x_0$ in the Tower. For these technical reasons
  our results are limited to $\mu$ a.e. $x_0$ for polynomially mixing
  billiards.

\end{rmk}

\begin{thm}\label{thm:mixing-condition_Young-Tower}
  For $1<\alpha<2$, relation~\eqref{eq:D9bis} holds for the system
  $(F,M,\mu_M)$ considered above and the observable
  $\Phi_2=\phi_2=\phi|_M$, $\phi(x)=d(x,x_0)^{-\frac{2}{\alpha}}$,
  for a.e. $x_0$.

  Therefore 
  \[
    \frac{1}{n^{1/\alpha}}\sum_{k=1}^n (\Phi_2\circ F^k - \mu_M(\Phi_2))
    \dto X_\alpha
  \]
  where $ X_\alpha$ is a stable law of index $\alpha$.
\end{thm}

\begin{proof}[Proof of Theorem~\ref{thm:mixing-condition_Young-Tower}]
  
  To obtain~\eqref{eq:D9bis}, by exponential decay of correlations of
  Lipschitz versus Lipschitz, we need only show that
\begin{equation}\label{eq:k-log-n}
  \lim_{\epsilon \to 0} \limsup_{n\to
    \infty}\frac{n}{b_n^2}\sum_{j=1}^{\lfloor k \log n
    \rfloor}\max\left\{0, \int_M \Phi_n \cdot \Phi_n\circ F^{j}d\mu_M\right\}=0, 
\end{equation}
where $k$ is independent of $n$ and $\epsilon$. To do this we approximate
$\Phi_n$ sufficiently close in the $L^1$ norm by a Lipschitz
function
$\Psi_n$, ($\|\Psi_n-\Phi_n\|_{L^1} \le \frac{1}{n^{2}}$ say) and then use
exponential decay of correlations, meaning that for
Lipschitz functions $f$ and $g$ on $M$,
\[
  |\mu_{M} (f \cdot g\circ F^j)-\mu_{M}(f)\mu_M (g)|\le
  C\theta^{j}\|f\|_{\Lip} \|g\|_{\Lip}
\]
for some $0 < \theta <1$. 

The first return map $F$ to the set $M$ has the structure of a uniformly
hyperbolic billiard system with discontinuities as described
in~\cite[Section 5]{Chernov_dispersing}.

Since $\mu_M$ is $F$-invariant, can rewrite the covariance
$\int \Phi_n\Phi_n\circ F^{j}d\mu_M$ as
$\mu_M(\phi_n \cdot \phi_n\circ F^{j})-[\mu_M(\phi_n)]^2$. Because
$\phi_2\in L^1(\mu_M)$, one can neglect the $[\mu_M(\phi_n)]^2$ terms
in~\eqref{eq:k-log-n} as their contribution is of order
\[
  \frac{n}{b_n^2} ( \mu_M(\phi_n))^2 \log{n} \le \frac{n}{b_n^2} (
  \mu_M(|\phi_2|))^2 \log{n} \sim ( \mu_M(|\phi_2|))^2 n^{1-\frac{2}{\alpha}}
  \log{n}
\]
and $\alpha < 2$.

Thus, it suffices to show
\begin{equation}\label{eq:D9sub}
  \lim_{\epsilon \to 0} \limsup_{n\to
    \infty}\frac{n}{b_n^2}\sum_{j=1}^{\lfloor k \log n \rfloor}\int_M
  |\phi_n| \cdot |\phi_n|\circ F^{j} d\mu_M=0 \text{\qquad with $\phi_n
    :=\phi_2\cdot \one_{\left\{|\phi_2| \le \epsilon b_n\right\}}, \Phi_n=\phi_n-\mu_M(\phi_n)$, see~\eqref{eq:phi_n}}
\end{equation}

Introduce
\begin{equation}\label{eq:U_n-mixing-proof}
  \frac {1}{2} < r < 1, \quad u_n:=b_n^r,  \quad U_n :=
  \{x\in M \;:\;|\phi_2(x)| \ge u_n\}.
\end{equation}
Since $\phi_2=\phi \one_M$ is in the domain of attraction of a stable law
with index $\alpha$ (see \eqref{eq:slow-variation1} in
Remark~\ref{rmk:usual-stable-description}),
\begin{equation*}
  \mu (U_n) =  u_n^{-\alpha} L(u_n).
\end{equation*}

From Karamata's Theorem~\ref{prop:karamata}, equation

\eqref{eq:slow-variation1} and that $u_n =b_n^{{r}}$, we have
\begin{equation}\label{eq:computation1}
  \int {\phi_n}^{2} d\mu = \int \phi^2 \cdot \onebr{|\phi|\le \epsilon b_n}
  d\mu \sim \frac{\alpha}{2-\alpha} (\epsilon b_n)^2  \mu (|\phi| > \epsilon
  b_n) = C_\alpha \epsilon^2 b_n^{2} (\epsilon b_n)^{-\alpha} L(\epsilon b_n)
\end{equation}
\begin{equation}\label{eq:computation2}
  \int_{U_n^c} {\phi}^{2} d\mu \sim \frac{\alpha}{2-\alpha} u_n^2
  \mu (|\phi| > u_n) = C_\alpha b_n^{2r} b_n^{-\alpha r} L(u_n)
\end{equation}
We decompose the sum of integrals in \eqref{eq:D9sub} as
${\rm (I)} + {\rm (II)} + {\rm (III)}$, where
\[ {\rm (I)} = \sum_{j= 1}^{\lfloor k \log n \rfloor} \int_{U_n \cap
    F^{-j}U_n} |\phi_n| \cdot |\phi_n|\circ 
F^{j} d\mu,
\]
\[ {\rm (II)} = \sum_{j= 1}^{\lfloor k \log n \rfloor} \int_{U_n \cap
    F^{-j}U_n^c} |\phi_n| \cdot |\phi_n|\circ F^{j} d\mu
\]
and
\[ {\rm (III)} = \sum_{j= 1}^{\lfloor k \log n \rfloor} \int_{U_n^c}
  |\phi_n| \cdot |\phi_n|\circ F^{j} d\mu.
\]

The sums (II) and (III) are dealt with as in the proof of ~\cite[Theorem
8.4]{CNT2025}. We repeat the calculations here for completeness. For (I) we
will use the Young Tower structure of $M$. Consider (II) and (III) first.

{For (III)}, using that $\mu$ is $F$-invariant, we have
\begin{equation}\label{eq:GM_III}
  \begin{aligned} 
    \int_{U_n^c} |\phi_n| \cdot |\phi_n| \circ F^j d\mu 
    \le \left( \int_{U_n^c} {\phi}^{2} d\mu \right)^{\frac 1 2}
    \left( \int \phi_n^2 \circ F^j d\mu \right)^{\frac 1 2} 
    = \left( \int_{U_n^c} {\phi}^{2} d\mu \right)^{\frac 1 2}
    \left(\int \phi_n^2  d\mu \right)^{\frac 1 2}.
  \end{aligned}
\end{equation}
Similarly, for (II),
\begin{equation}\label{eq:GM_II}
  \begin{aligned}
    \int_{U_n \cap F^{-j}U_n^c} |\phi_n| & \cdot
    |\phi_n| \circ F^{j} d\mu
     \le \left( \int \phi_n^2 d\mu \right)^{\frac 1 2}
      \left( \int_{F^{-j}U_n^c} {\phi}^{2} \circ F^{j} d\mu
      \right)^{\frac 1 2}\\
    & =
      \left( \int \phi_n^2 d\mu \right)^{\frac 1 2}
      \left( \int {(\phi}^{2} \cdot \one_{U_n^c})\circ F^jd\mu
      \right)^{\frac 1 2}
      =
      \left( \int \phi_n^2 d\mu \right)^{\frac 1 2}
      \left( \int {\phi}^{2} \cdot \one_{U_n^c}d \mu
      \right)^{\frac 1 2}\\
    &
      =
      \left( \int \phi_n^2 d\mu \right)^{\frac 1 2}
      \left( \int_{U_n^c} {\phi}^{2} d \mu
      \right)^{\frac 1 2}
  \end{aligned}
\end{equation}
By \eqref{eq:computation1} and \eqref{eq:computation2}
we obtain
\begin{equation*}
  \left(\int \phi_n^2 d\mu \right)^{\frac 1 2}
  \left( \int_{U_n^c} {\phi}^{2} d \mu
  \right)^{\frac 1 2} \le C_\alpha \epsilon^{1-\frac{\alpha}{2}}
  b_n^{(1-\frac{\alpha}{2})(1+r)} L(\epsilon b_n)^{1/2} L(b_n^r)^{1/2}
\end{equation*}
By \eqref{eq:tail1} and \eqref{eq:slow-variation1},
\[
  n\sim 1/\mu (|\phi|> b_n) = b_n^\alpha L(b_n)^{-1}
\]
which gives
\begin{equation*}
  \frac{n}{b_n^2}[{\rm (II)} + {\rm (III)}] \le 2C_\alpha k \epsilon^{1-\frac{\alpha}{2}}
  b_n^{-(1-\frac{\alpha}{2})(1-r)}\log{n} \cdot \left(\frac{L(\epsilon b_n)
      L(b_n^r)}{L(b_n)^2}\right)^{1/2} 
\end{equation*}
Since $L$ is slowly varying, it grows slower than any power (see
e.g.~\cite[Lemma~A.1]{CNT2025}),
so, because $r < 1$,

\begin{equation}\label{eq:II+III}
  \limsup_{n\to \infty}\frac{n}{b_n^2} [{\rm (II)} + {\rm (III)}] =0
\end{equation}

For the sum (I) we have to estimate
\begin{equation}\label{eq:mixing-without-removing-mean}
  \sum_{j= 1}^{\lfloor k \log n \rfloor} \int_{U_n \cap F^{-j}U_n}
  |\phi_n| \cdot |\phi_n|\circ F^{j} d\mu.
\end{equation}

To do this we will use the fact that $M$ has the structure of a Young Tower
for the uniformly hyperbolic billiard system with discontinuities, as
described in~\cite[Lemma 10]{Zhang2017}. We will also use results on the
genericity of points without short returns
from~\cite{GuptaHollandNicol2011}. We may suppose that $x_0\in \Delta_0$,
the base of the Young Tower, which is a ``hyperbolic horseshoe''. We
briefly describe the relevant features of this Tower construction. The base
$\Delta_0$ is a rectangle of local stable manifolds and local unstable
manifolds. There is a countable partition $\{A_k\}$ of $\Delta_0$ into
$s$-subsets such that each $W^u$ is partitioned into $A_k \cap W^u$ and
each $A_k$ contains $W^s(x)\cap \Delta_0$ for each $x\in A_k$. There exists
a return time $\tau(k)$ of $A_k$ such that
$F^{\tau(k)} A_k \subset \Delta_0$ and makes a full crossing of $\Delta_0$,
in sense that the local unstable manifolds are expanded by $G:=F^{\tau}$ to
traverse fully from one `vertical' boundary of $\Delta_0$ to the other. The
`vertical' boundaries of $\Delta_0$ consist of two local stable leaves. We
define $\tau(x)=\tau(k)$ if $x\in A_k$. The Tower is exponentially mixing
and $\mu (\tau >n)\le C\theta^n$ for some $0<\theta <1$.

 The local unstable manifolds in each $A_k$ expand uniformly in the
 sense that $|F^j W^u| \ge \lambda^j |W^u|$ for some $\lambda>1$ until
 they make a full crossing at time $\tau(k)$. The Tower $\Delta$
 consists of the set
 \[
   \{(x,j): 0\le j \le \tau(x) -1, \ x\in A_k, \ k=1,2,\ldots\}
 \]
 together with the map 
 \[
   f(x,\ell)= \left\{ \begin{array}{ll}
         (x,\ell+1)& \mbox{if $\ell < \tau(x)-1$};\\
        (Gx,0) & \mbox{if $\ell =\tau(x)-1$}.\end{array} \right.
 \] 
  
  We let  
  \[
    C_k:=\{(x,j): 0\le j \le \tau(x) -1,\ x\in A_k\}
  \]
  denote the column over the set $A_k \subset \Delta_0$.
  
  The projection $\pi :\Delta \to M$ projects the points of the Tower onto
  the manifold $M$,
  \[
    \pi(x,\ell)=F^\ell (x).
  \]
  The usual bounded distortion and absolute continuity estimates hold as
  in~\cite{Young_98}.
  
  In~\cite[Section 4.1.2.]{GuptaHollandNicol2011} it is shown that for
  $\mu$ a.e. $x_0$, there exists $\delta>0$ such that
  $\mu(U_n \cap F^{-j} U_n)=O(b_n^{-\alpha r(1+\delta)})$ for
  $j=1,\dots ,(\log n)^5$.
  
  Therefore for large $n$
  \[
  (I)\le (k\log n) b_n^2 \epsilon^2 b_n^{-\alpha r(1+\delta)}
  \]
  and 
  \[
 \frac{n}{b_n^2} (I)\le n^{1-r(1+\delta)}\epsilon^2
 \]
 Taking $r>\frac{1}{(1+\delta)}$ gives
 $\lim_{n\to \infty} \frac{n}{b_n^2} (I)\to 0$.

 Together with \eqref{eq:II+III}, this shows that condition
 \eqref{eq:D9sub} is satisfied.
\end{proof}

\section{Proofs}\label{sec:proof-main}

\subsection{Proof of Theorem~\ref{main1}, case (a):
  $\frac{1}{\alpha}> \gamma$.}\ \null

This is the most straightforward case.

As an application of the work of~\cite{Pene-Saussol2020} and by verifying
the condition \eqref{eq:D9bis}
for generic $x_0$ we have shown that
\[
  (\bar{R}n)^{-\frac{1}{\alpha}} \sum_{j=0}^{n}(\Phi_2\circ F^j-\mu_M
  [\Phi_2 ])
\]
converges on $M$ to a stable law $X_{\alpha}$ of index $\alpha$ for
$0<\alpha<2$, $\alpha\not =1$. We lift this limit law from the induced
systems $(F,M,\mu_M)$ to $(\tilde{Q},T,\mu)$, using~\cite[Theorem
4.6]{Gouezel_Doubling} (see section \ref{gouezel-lifting}). 

Meanwhile,
the observable $\phi_1-\mu[\phi_1]$ on $\tilde{Q}$ satisfies a stable law of index
$\frac{1}{\gamma}$ with normalization $n^{\gamma}$  where $\gamma <\frac{1}{\alpha}$ if $I_{\tilde{\phi_1},i}\not =0$ for some $i\in \Jmax$, or 
its variance grows slower than $n^{2\gamma-\epsilon}$, in either
case  hence its contribution vanishes under scaling by $n^{\frac{1}{\alpha}}$.

We verify condition (b) of Gou\"{e}zel's
Proposition~\ref{gouezel-lifting} with

$\alpha(n)=n^{\gamma}$ and $B_n=(n\bar{R})^{\frac{1}{\alpha}}$. If
$1<\alpha<2$ then take $A_n=n\mu_M (\Phi_2)$, if $0<\alpha<1$ take $A_n=0$.
Recall that $\mu_M (\Phi_2)$ is the expectation on $(M,\mu_M)$, so
$\mu_M (\Phi_2)=\mu_M (R) \mu (\phi_2)$ because $\mu_M (R) = 1/\mu(M)$.

Note that $R$ satisfies a stable law of index $\frac{1}{\gamma}$ under $F$~\cite[Corollary 4.3]{TK-dynamical}.
 Indeed, the return-time function $R$ is
constant on partition elements of $M$ and hence measurable with respect to
the partition on $M$. The distribution of $R$ is in the domain of
attraction of a stable law of index $\frac{1}{\gamma}$. Hence,
by~\cite[Corollary 4.3]{TK-dynamical},
\[
  \frac{1}{n^{\gamma}} \sum_{j=0}^{n} (R\circ F^j-\mu_M (R))
\]
converges on $M$ to a stable law of index $1/\gamma$, and is therefore a
tight sequence.

When $1<\alpha<2$ $\Phi_2$ is in $L^1$ and satisfies the Birkhoff
ergodic theorem. If we let $m=n^{\gamma}$ then
\[
  \frac{1}{n^\frac{1}{\alpha}} \sum_{j=1}^{n^{\gamma}} \Phi_2\circ F^j =
  \frac{1}{m^{1/(\gamma \alpha)}} \sum_{j=1}^{m} \Phi_2 \circ F^j
\]
and since $1/(\gamma\alpha)>1$ the latter sum converges to zero almost
surely. Therefore
\[
  \max_{k \leq n^{\gamma}} \frac{1}{n^{\frac{1}{\alpha}}} S_{k}(\Phi_2) \dto 0
\]

If $0 < \alpha < 1$ then $\Phi_2 > 0$ so
\[
  \max_{j \le n^{\gamma}} \frac{S_{j}(\Phi_2)}{n^{1/\alpha}} =
  \frac{S_{n^{\gamma}}(\Phi_2)}{n^{1/\alpha}}\dto 0
\]
because $n^{\gamma} < n$ and ${S_{n}(\Phi_2)}/{n^{1/\alpha}} \dto X_{\alpha}$.

Therefore, we can apply Proposition~\ref{gouezel-lifting}, part (b) to
conclude that on  $\tilde{Q}$
\begin{equation*}
  n^{-\frac{1}{\alpha}}
  \sum_{j=0}^{n-1} (\phi_2 \circ T^j - \mu (\phi_2)) \dto X_{\alpha}
\end{equation*}

Now $\phi_1-\mu (\phi)$ satisfies a stable law with scaling $n^{\gamma}$  if $I_{\tilde{\phi},i}\not =0$ for
some $i\in \Jmax$ or its Birkhoff sum converges in distribution to zero under scaling by $n^{-\gamma}$.
In either case the scaled Birkhoff sum
$n^{-\frac{1}{\alpha}} \sum_{j=0}^{n}[\phi_1 \circ T^j -\mu (\phi_1)]$
converges in distribution to zero.

This proves that
\[
  n^{-\frac{1}{\alpha} } \left(\sum_{j=1}^{n} (\phi_1 \circ T^j -
    \mu(\phi_1)) + (\phi_2 \circ T^j - \mu(\phi_2)) \right)
  =n^{-\frac{1}{\alpha} } \sum_{j=1}^{n} (\phi \circ T^j - \mu(\phi)) \dto
  X_{\alpha}
\]
where $X_{\alpha}$ has a stable distribution with index $\alpha$.

\subsection{Proof of Theorem~\ref{main1}, case (b):
  $\gamma > \frac{1}{\alpha}$,
  $I_{\tilde{\phi},i} \not = 0 \text{ for some }i\in \Jmax$}\ \null

This result may be proved in the same way as   Theorem~\ref{main2}, but we give a more
elementary proof that does not rely on Theorem~\ref{distribution}.
By \cite{Jung_Pene_Zhang2020}, if we induce  the  observable $\tilde{\phi_1}(x)=\phi_1(x)-\mu (\phi_1)$ on $M$
it has the same limit law under scaling by $n^{-\gamma}$ as 
\[
\sum_{i\in J} \frac{I_{\tilde{\phi_1},i}}{I_i} (R_i (x)-\mu_{M} (R_i))
\]

 We  now show that it is possible to 
lift the $\alpha$-stable law of $\Phi_2$ on  $(F,M,\mu_M)$ given by
Theorem~\ref{thm:mixing-condition_Young-Tower} to an $\alpha$-stable law of  $\phi_2$ on $(T,\tilde{Q},\mu)$.  Under the norming  $n^{-\gamma}$  the contribution of this  Birkhoff sum tends to zero in distribution.

Since $\gamma>\frac{1}{\alpha}$ we must have  $\alpha \in (1,2)$ and thus $\phi_2$
is integrable. Hence for any integer $P$
\begin{equation}\label{forward}
  \sup_{k<PN}\frac{1}{N}\sum_{j=1}^k (\phi_2\circ T^j -\mu ( \phi_2) ) \dto 0
\end{equation}

Since $T$ is invertible  and $T^{-1}$ preserves $\mu$ we also have
\begin{equation}\label{backward}
  \sup_{k<PN}\frac{1}{N}\sum_{j=1}^k (\phi_2\circ T^{-j} -\mu ( \phi_2))
  \dto 0
\end{equation}

Moreover, since both $T$ and $T^{-1}$ preserve $\mu$, for any $N$ the
distribution of
\[
  \sup_{k<PN}\frac{1}{N}\sum_{j=1}^k (\phi_2\circ T^j -\mu (\phi_2))
\]
is the same as the distribution of 
\[
  \sup_{k<PN}\frac{1}{N}\sum_{j=n\bar{R}}^{k+n\bar{R}} (\phi_2\circ T^j
  -\bE[\phi_2])
\]
and the distribution of
\[
  \sup_{k<PN}\frac{1}{N}\sum_{j=1}^k (\phi_2\circ T^{-j} -\mu (\phi_2))
\]
is the same as the distribution of 
\[
  \sup_{k<PN}\frac{1}{N}\sum_{j=n\bar{R}}^{k+n\bar{R}} (\phi_2\circ T^{-j}
  -\mu (\phi_2 ))
\]

We consider the following quantities defined pointwise on $M$ (recall that
$\Phi_2=\phi_2|_M$ and $\mu_M=\bar{R}\mu$) .

Case 1: $n\bar{R}>R_n(x)$.

\[
\sum_{j=1}^n
  (\Phi_2\circ F^j(x)-\mu_M (\Phi_2))
  \]
  \[
  =(\sum_{j=1}^{R_n(x)} \phi_2\circ T^j )-n\bar{R}\mu (\phi_2)
  \]
  \[
  =\sum_{j=1}^{n\bar{R}} (\phi_2\circ T^j -\mu (\phi_2))
  \]
  \[
  -\sum_{j=R_n(x)+1}^{\bar{R}n} (\phi_2\circ T^j -\mu (\phi_2)) 
  \]
  \[
  +\mu ( \phi_2) (R_n (x)-n\bar{R})
  \]
  Case 2: $R_n(x)>n\bar{R}$.

\[
\sum_{j=1}^n
  (\Phi_2\circ F^j(x)-\mu_M (\Phi_2))
  \]
  \[
  =(\sum_{j=1}^{R_n(x)} \phi_2\circ T^j )-n\bar{R}\mu (\phi_2)
  \]
  \[
  =\sum_{j=1}^{n\bar{R}} (\phi_2\circ T^j -\mu (\phi_2))
  \]
  \[
  +\sum^{j=R_n(x)}_{\bar{R}n+1} (\phi_2\circ T^j -\mu (\phi_2)) 
  \]
  \[
  +\mu (\phi_2) (R_n (x)-n\bar{R})
  \]
  
  Thus

\[
  \sum_{j=1}^{n\bar{R}}(\phi_2\circ T^j(x) -\mu (\phi_2)) =\sum_{j=1}^n
  (\Phi_2\circ F^j(x)-\mu_M (\Phi_2) ) + V_n(x) -\mu (\phi_2) (R_n (x)- n\bar{R})
\]
where 
\begin{equation}\label{eq:V_n-expression}
  V_n(x)=\begin{cases} \sum_{j=R_n(x)+1}^{n\bar{R}} (\phi_2\circ T^j(x)
    -\mu (\phi_2))
    &\text{\qquad if $R_n(x) \le n \bar{R}$}\\
    -\sum_{j=n\bar{R}+1}^{R_n(x)} (\phi_2\circ T^j(x) -\mu (\phi_2) )
    &\text{\qquad if $R_n(x) > n \bar{R}$}
  \end{cases}
\end{equation}

The quantity 
$-\mu (\phi_2) (R_n (x)- n\bar{R})$ is the Birkhoff sum of the observable $-\mu (\phi_2)(R(x)-\mu_M(R))$ under $F$.

Furthermore by~\eqref{constant}, if $\psi=-\mu (\phi_2)$ then
\[
\sum_{i\in J} \frac{I_{\psi,i}}{I_i} (R_i (x)-\mu_{M} (R_i))=-\mu (\phi_2) (R(x)-\mu_M(R))
\]

 Recall  the induction of the  observable $\tilde{\phi_1}(x)=\phi_1(x)-\mu (\phi_1 )$ on $M$ has the same limit law as 
\[
\sum_{i\in J} \frac{I_{\tilde{\phi_1},i}}{I_i} (R_i (x)-\mu_{M} (R_i))
\]
Furthermore
\[
\sum_{i\in J}  \frac{I_{\tilde{\phi_1},i}}{I_i} (R_i (x)-\mu_{M} (R_i)) -\mu (\phi_2) (R(x)-\mu_M(R))
\]
\[
=\sum_{i\in J}  \frac{I_{\tilde{\phi},i}}{I_i} (R_i (x)-\mu_{M} (R_i)) 
\]
where $\tilde{\phi}=\phi-\mu(\phi)$.
Under the condition $I_{\tilde{\phi},i}\not =0$ for some $i\in \Jmax$, by~\cite{Jung_Pene_Zhang2020}
the induced observable $\sum_{i\in J}  \frac{I_{\tilde{\phi},i}}{I_i} (R_i (x)-\mu_{M} (R_i))$ lifts to give a stable law
of index $\frac{1}{\gamma}$.

 We will now show that the other contribution from  the induction of $\phi_2-\mu (\phi_2)$
on $M$ converges to zero in distribution under the scaling $n^{-\gamma}$, that is $n^{-\gamma}\sum_{j=1}^n [(\Phi_2\circ F^j(x)-\mu_M (\Phi_2) )] + n^{-\gamma} V_n(x)$ converges to zero in distribution.

As in Case (a), with $R_n=\sum_{k=0}^{n-1}R\circ F^k$ and
$\bar{R}=\mu_M (R)$,
\[
  (R_n(x)-n\bar{R})/n^{\gamma}\dto X_{\frac{1}{\gamma}} \qquad \text{ on
    $M$}
\]
where $X_{1/\gamma}$ is a stable law of index $1/\gamma$.

Since $T$ preserves $\mu$, for any $n$ sufficiently large
\begin{equation}\label{addition}
 \mu (x\in M: |R_n(T^{-n\bar{R}})-\bar{R}n|>Pn^{\gamma} )<\epsilon
\end{equation}

 Therefore, for
any $\epsilon >0$ there exists an integer $P$ such that for all
sufficiently large $n\ge N$,
\[
\mu (x\in M: |R_n(x)-\bar{R}n|>Pn^{\gamma} ) <\epsilon
\]

The reason we consider
$\mu(x\in M: |R_n(T^{-n\bar{R}})-\bar{R}n|>Pn^{\gamma})$ is to bound
$|R_n(x)-\bar{R}n|$ in
\[
  \sum_{j=n\bar{R}}^{R_n(x)} (\phi_2\circ T^j(x) -\mu (\tphi_2) )
\]
and then run forwards in time under $T$ to estimate~\eqref{forward} and to
bound $|R_n(x)-\bar{R}n|$ in
\[
\sum_{j=R_n(x)}^{n\bar{R}} \phi_2\circ T^j -\mu ( \phi_2)
\] 
and then run backwards in time to estimate~\eqref{backward}.

By~\eqref{forward}, \eqref{backward} and \eqref{addition}
\begin{equation}\label{eq:V_n-estimate}
  \frac{1}{n^{\gamma}} V_n\dto 0.
\end{equation}

Recall that $\bE_M[\phi_2] =\bar{R}\mu (\phi_2)$. By
Theorem~\ref{thm:mixing-condition_Young-Tower} a stable law of index
$\alpha$ holds for $\Phi_2$ on the induced system, that is
\[
  \frac{1}{n^{\frac{1}{\alpha}}} \sum_{j=1}^n (\Phi_2\circ
  F^j-\mu_M (\Phi_2) ) \dto X_{\alpha} \text{ \quad on $M$}
\]

Since $\gamma>\frac{1}{\alpha}$, together with \eqref{eq:V_n-estimate} this
implies that
\[
  \frac{1}{n^{\gamma}} \sum_{j=1}^n (\phi_2\circ T^j-\mu (\phi_2) ) \dto 0
  \text{ \quad on $\tilde{Q}$}
\]

Hence 
\[
  n^{-\gamma} \sum_{j=0}^{n} (\phi \circ T^j -
  \mu (\phi ) )\dto X_{\frac{1}{\gamma}}
\]
where $X_{\frac{1}{\gamma}}$ is the distribution determined by $\phi$,
according to~\cite[Theorem 2.1]{Jung_Pene_Zhang2020}.

\subsection {Proof of Theorem~\ref{main2}: 
  $\gamma > \frac{1}{\alpha}$, and $\int \phi d\mu=0$ with $\phi\equiv 0$
  on $U$}\ \null

We assume that the induced version of $\phi$ on $M$ satisfies a stable law of index $\alpha$, that is
if 
$\Phi (x)=\sum_{j=0}^{R(x)-1}\phi\circ T^j (x)$ then 

\[
n^{-\frac{1}{\alpha}}\sum_{j=1}^n (\Phi\circ F^j-\mu_{M} (\Phi ))\dto X_{\alpha}
\]
where $X_{\alpha}$ is a stable law of index $\alpha$.

We furthermore assume, as in Theorem~\ref{main2}, that each cusp $P_i$,
$i=1,\dots,q$, has a neighborhood $O_i$ such
that $\phi \equiv 0$ on $U=\cup_{i=1}^{q} O_i$.

We will show that the stable law may be lifted. In this setting 
 in our former decomposition $\phi_1\equiv 0$ so we need only consider
$\phi=\phi_2$, with the normalization $\int \phi_2 d\mu=\mu (\phi_2 )=0$. This case
does not reduce to Case (b) since $\phi_1$ (now zero) no longer satisfies
a stable law of index $\frac{1}{\gamma}$; we will show that we may lift the stable law
of index $\alpha$ obtained for $\Phi_2$ on $M$ to a stable law of index $\alpha$ for $\phi_2$ on $\tilde{Q}$.

The question becomes one of lifting the induced
stable law of index $\alpha$ to the whole space when the return time
function $R$ is slow and of stable index $\frac{1}{\gamma}< \alpha$.
This would be an application of the continuous mapping theorem in the case
that we have a functional stable limit law, but we give a more general
argument here. To lift we will use the fact that $(F, M, \mu_M)$ is
exponentially mixing with $B_1=B_2=\Lip(\tilde{Q})$ and we use  Theorem~\ref{distribution}.

Again we  consider the following quantities defined on $M$,
\[
  \sum_{j=1}^{n\bar{R}}\phi_2\circ T^j(x) =\sum_{j=1}^n \phi_2\circ
  F^j(x) +V_n(x)
\]
where $V_n$ is 
\begin{equation}
  V_n(x)=\begin{cases} \sum_{j=R_n(x)+1}^{n\bar{R}} \phi_2\circ T^j(x)
    &\text{\qquad if $R_n(x) \le n \bar{R}$}\\
    -\sum_{j=n\bar{R}+1}^{R_n(x)} \phi_2\circ T^j(x) 
    &\text{\qquad if $R_n(x) > n \bar{R}$}
  \end{cases}
\end{equation}
as in \eqref{eq:V_n-expression}.

Since 
\[
  (R_n(x)-n\bar{R})/n^{\gamma}
\]
converges in distribution on $M$ to a stable law of index
$\frac{1}{\gamma}$, given $\epsilon >0$ there exists an integer $P$ such
that for all sufficiently large $n\ge N$,
\[
\mu_{M} (x\in M: |R_n(x)-\bar{R}n|>Pn^{\gamma} ) <\epsilon
\]
Now $\mu_M$ is preserved under $F$  (note $F^n  (x)=T^{R_n(x)}(x)$)  and 
$\mu$ is preserved under $T$. Thus we may assume that except for a set of $\mu_M$ measure at most $\epsilon$, 
$|R_n(x)-\bar{R}n|<Pn^{\gamma}$.

Let $\delta>0$, then 
 \[
 \mu_{M} \{ |V_n(x)|>\delta n^{\frac{1}{\alpha} } \}
 \]
 \[
\le \mu_M \{ x\in M:  \max_{0\le j\le  Pn^{\gamma} }  n^{-\frac{1}{\alpha}}  |\sum_{i=0}^{j} (\phi_2\circ T^j(x) |>\delta\} +\epsilon
    \]

Let $m=Pn^{\gamma}$ for $P$ and $n$ large. Then
$n^{\frac{1}{\alpha}}=(\frac{m}{P})^{\frac{1}{\alpha\gamma}}$. We rewrite
\[
  \frac{1}{n^{1/\alpha}} |\max_{k<Pn^{\gamma}}\sum_{j=0}^{k} \phi_2 \circ
  T^j|
 = \frac{P^{\frac{1}{\alpha\gamma} }}{m^{\frac{1}{\alpha\gamma}}}
  |\max_{k<m}\sum_{j=0}^{k} \phi_2 \circ T^j|
\]

For fixed $x\in M$ and $m$, let $k=k(x,m) < m$ be such that $F^k(x)$ is the
last entry of $x$ in $M$, i.e. $k$ is the largest integer satisfying
$R_k(x) < m$. Recall that $\phi_2=0$ outside $M$. Then
\[
  \sum_{j=0}^{m} \phi_2 \circ T^j(x) =\sum_{j=0}^{k(x,m)-1}
  \Phi_2(x)\circ F^j (x)
\]

Thus
\[
  \frac{1}{m^{\frac{1}{\alpha\gamma}}} |\max_{k<m}\sum_{j=0}^k \phi_2\circ
  T^j(x)| \le \frac{1}{\frac{1}{m^{\alpha\gamma}}
  }|\max_{k<m}\sum_{j=0}^{m} \Phi_2\circ F^j|
\]

Since $\gamma <1$, $\frac{1}{\gamma\alpha} > {\frac{1}{\alpha}}$. By
Theorem~\ref{distribution}
\[
  \frac{P^{\frac{1}{\alpha\gamma} }}{m^{\frac{1}{\alpha\gamma}}}
  \max_{k<m}\sum_{j=0}^{k} \Phi_2(x)\circ F^j \dto
0
\]
Thus  $\mu_{M} \{ |V_n(x)|>\delta n^{\frac{1}{\alpha} } \}\to 0$ and since $\epsilon>0$ and 
 $\delta>0$ were arbitrary this concludes the proof.

\subsection{Proof of Proposition~\ref{thm:intermittent}}
The cases $(a)$ and $(b)$ were proved in~\cite{CNT2025} and we list them here for completeness. 
  The  new result is (c). The same proof as in Theorem~\ref{main2} holds.  We  may use Theorem~\ref{distribution}  as we have exponential mixing for the induced
  system (which is Gibbs-Markov) in the Banach spaces $B_1=BV$ versus $B_2=L_{\infty}$.

\subsection{Proof of Theorem~\ref{cusp_intermittent}}\ \null

    We induce the LSV map \eqref{IM} with $\gamma\in (0,1)$ onto
  $Y=[1/2, 1]$. Let $C_j=(y_{j+1},y_j)$ be the partition element in
  $[1/2,1]$ corresponding to a first return time of $j$. Let
  $x_j=T_1^{-j} (1/2) \in[0,1/2)$be the preimage of $1/2$  under the first branch of the map, so that
  $x_j\sim j^{-\frac{1}{\gamma}}$ and  $x_{j-1}=T(y_j)$.

  Let $\phi(x)=x^{-\frac{1}{\alpha}} -\mu_{\gamma} (x^{-\frac{1}{\alpha}})$ (if $x^{-\frac{1}{\alpha}}$ is integrable) or $x^{- \frac{1}{\alpha}} $ if not.
 
  By the mean value theorem, for $0\le j\le n$,
  \[
    |\phi(x_{j+1})- \phi(x_j)|\le (j^{-\frac{1}{\gamma}})^{-\frac{1}{\alpha}-1}
    |j^{-\frac{1}{\gamma}}-(j+1)^{-\frac{1}{\gamma}}|\sim
    j^{\frac{1}{\gamma\alpha}-1}
  \]
  as, for large $j$, 
  $|j^{-\frac{1}{\gamma}}-(j+1)^{-\frac{1}{\gamma}}| \sim
  j^{-\frac{1}{\gamma}-1}$.

Consider the locally constant function, defined on $[1/2,1]$, 
\[
  \Psi(x) =\sum_{j=1}^n \phi(x_j) \qquad \text{for
    $x\in C_n=(y_{n+1},y_n)\subset [1/2, 1]$}
\]
Note that  $\Psi(x)=\Psi (y_n)$ for all $x\in (y_{n+1},y_n)$.
We extend $\Psi$ to $[0,1]$ by $\Psi(x)=\phi(x_j)$ for $x\in (x_{j+1},x_j]$.
Similarly, let
\[
  \Phi(x)=\sum_{j=1}^n \phi(T^j(x)) \qquad \text{for $x\in C_n$ }
\]
be $\phi$ induced on the base element $C_n$ of the tower. The function $\Phi$ is not constant on 
$C_n$.
 
 We calculate, for $x\in C_n$,
 \[
   |\Psi (x)-\Phi (x)|\le \sum_{j=1}^n j^{\frac{1}{\alpha\gamma}-1}\le C
   n^{\frac{1}{\alpha\gamma}} 
 \]
 
 On each base element $C_n$ of the tower we have the constant function
 $\Psi$. We calculate
 \[
   \Psi(y_n)=\sum_{j=1}^n (j^{-\frac{1}{\gamma}})^{-\frac{1}{\alpha}}\sim
   n^{1+\frac{1}{\alpha\gamma}}
 \]
 so if $x\in C_n$ then
 \[
   \Psi(x) \sim n^{1+\frac{1}{\alpha\gamma}}
 \]
 Furthermore each $C_n$ has Lebesgue measure
 $m(C_n)=n^{-1-\frac{1}{\gamma}}$. The invariant measure $\mu$ of the LSV
 map is equivalent to Lebesgue measure on the base of the tower.
 
 We solve for $b_n$, i.e. to have scaling $nm(\Psi >b_n)=1$,
 \[
 m(C_j: \Psi_{|C_j} >b_n)=\frac{1}{n}
 \]
 \[
 =m(j^{1+\frac{1}{\alpha\gamma}} > b_n)=\frac{1}{n}
 \]
 or equivalently we solve for 
 \[
 m(j>(b_n)^{\frac{\alpha\gamma}{1+\alpha\gamma}})=\frac{1}{n}
 \]
 Since for any $t$, 
 \[
 \sum_{j=t}^{\infty} j^{-1-\frac{1}{\gamma}}\sim t^{-\frac{1}{\gamma}}
 \]
 we calculate
 \[
 (b_n)^{\frac{\alpha\gamma}{1+\gamma\alpha}(-\frac{1}{\gamma})}=n^{-1}
 \]
 which yields
 \[
 b_n=n^{\frac{1}{\alpha}+\gamma}
 \]

We let $g$ be a  function on the base which is constant on each $C_j$ and
has value 
\[
g|_{C_j}=g(j):=j^{\frac{1}{\alpha\gamma}}
\]

Note that if $x\in C_j$ then
 \[
   |\Phi(x)-\Psi(x)|\le g(j)
 \]
 
 If we compute the stable index of $g$ we find
 \[
   m(g>t) \sim m(j>t^{\alpha\gamma}) =\sum_{j=t^{\alpha\gamma}}^{\infty}
   j^{-1-\frac{1}{\gamma}}=t^{-\alpha}
 \]
Since $m(C_j)=j^{-1-\frac{1}{\gamma}}$ and
$g(j)\le j^{\frac{1}{\alpha\gamma}}$ we see that
\[
  \sum_{j=1}^{\infty} g(j)m(C_j) < \infty \qquad \text{if $1<\alpha<2$}
\]

Let
$F: [1/2,1]\to [1/2,1]$ be the Gibbs-Markov base map induced on the base of
the tower. We write $\Lambda:=[1/2,1]$ and note that $F$ has an invariant Lebesgue
equivalent measure $m_{\Lambda}$ on $\Lambda$ for which $F$ has exponential
decay of correlations in $B_1=BV$ versus $B_2=L^{\infty}$.

 If $\frac{1}{\alpha}+\gamma<1$ then $\Psi$, $\Phi$ and $g$ are integrable.
 
The function $ G(x):=(\Phi (x)-m_{\Lambda} (\Phi) )-(\Psi(x)-m_{\Lambda}(\Psi))$ has stable index $\alpha$.

We will show, via Theorem~\ref{distribution}, that $n^{-(\frac{1}{\alpha}+\gamma)}\sum_{j=0}^{n-1} G\circ F^j\dto 0$.
Since $n^{-(\frac{1}{\alpha}+\gamma)}\sum_{j=0}^{n-1}(\Psi(F^j x)-m_{\Lambda}(\Psi))\dto X_{\frac{\alpha}{1+\gamma\alpha}}$
by Slutsky's Theorem $n^{-(\frac{1}{\alpha}+\gamma)}\sum_{j=0}^{n-1}\Phi (F^j x)-m_{\Lambda} (\Phi)\dto X_{\frac{\alpha}{1+\gamma\alpha}}$.
 Then we will lift the induced stable law to the ambient system.

 First we note that on $C_j$, $|G(x)|<j^{\frac{1}{\alpha\gamma}}$. Let $\epsilon>0$, we solve for $j^{\frac{1}{\alpha\gamma}}>n^{\frac{1}{\alpha}+\epsilon}$
 and obtain $j> n^{\gamma+\epsilon\gamma\alpha}$. We truncate $G(x)$ to
 \[ \psi_n (x) = \left\{ \begin{array}{ll}
         G(x) & \mbox{if $x \in C_j, j<n^{\gamma+\epsilon\gamma\alpha}$};\\
        0 & \mbox{if otherwise}.\end{array} \right. \]

        Since $\psi_n \in BV[1/2,1]$ in Theorem~\ref{distribution} we define $f_n=\psi_n-\mu(\psi_n)$ and conditions (a), (c) and (d) of Theorem~\ref{distribution}
        are trivially satisfied.  It is easy to calculate that $\|f_n\|_{BV}\le \sum_{j=1}^{n^{\gamma+\epsilon\gamma\alpha}} j^{\frac{1}{\alpha\gamma}}\sim n^{\frac{1}{\alpha}+\gamma+\epsilon(1+ \gamma\alpha)}$. So we take $\kappa=\frac{1}{\alpha}+\gamma+\epsilon(1+ \gamma\alpha)$ 
        in Theorem~\ref{distribution} and conclude that $n^{-(\frac{1}{\alpha}+\gamma)}\sum_{j=0}^{n-1} G\circ F^j\dto 0$. The techniques of the proof
        of Theorem~\ref{main2} allow us to lift the stable law for the induced system, $n^{-(\frac{1}{\alpha}+\gamma)}\sum_{j=0}^{n-1}\Phi (F^j x)-m_{\Lambda} (\Phi)\dto X_{\frac{\alpha}{1+\gamma\alpha}}$, to a stable law for $\phi$ of the same index.

Finally if $\frac{1}{\alpha}+\gamma >1$ then 
no centering for $\Psi$ or $\Phi$  is needed. If $\alpha\in (1,2)$ then $g$ is integrable and since
$b_n=n^{\frac{1}{\alpha}+\gamma}$,
\[
  b_n^{-1} \sum_{j=1}^n g\circ F^j \to 0 \qquad \text{$m$ a.e.}
\]
Hence 
\[
b_n^{-1} \max_{k\le n} \sum_{j=1}^k g\circ F^j \to 0 \qquad \text{$m$ a.e.}
\]
so 
\[
b_n^{-1} \max_{k\le n} \sum_{j=1}^k |\Phi\circ F^j -\Psi\circ F^j| \to 0 \qquad \text{$m$ a.e.}
\]
Thus by (b) of Proposition~\ref{gouezel-lifting}.
\[
b_n^{-1} \sum_{j=1}^n (\Phi\circ F^j -\Psi\circ F^j)
\]
lifts to $[0,1]$ and converges to zero in distribution. As 
\[
b_n^{-1} \sum_{j=1}^n \Psi \circ T^j
\]
converges to a stable law of index $(\frac{1}{\alpha}+\gamma)^{-1}$ and 
 \[
b_n^{-1} \sum_{j=1}^n (\Psi\circ T^j -\phi\circ T^j)\dto 0
\]
by Slutsky's theorem 
\[
b_n^{-1} \sum_{j=1}^n \phi\circ T^j
\]
converges to a stable law of index $(\frac{1}{\alpha}+\gamma)^{-1}$.

If $\alpha \in (0,1)$ then $g$ is not integrable but $\int g^{\alpha-\epsilon} dm <\infty$ for any $\epsilon>0$, and then 
\[
n^{-1/(\alpha-\epsilon)}\sum_{j=1}^{n} g\circ F^j \to 0  \qquad \text{$m$ a.e.}
\]
by Aaronson~\cite[Proposition 2.3.1]{Aaronson}.  The rest of the argument proceeds as in the previous case.

  \section{Discussion and open problems.}
  
  The authors have work in progress to determine the stable limit law of a heavy tailed  observable maximized at 
  the cusp of a Machta billiard table. However, preliminary work suggests that the stable index
  will not be equal to $(\frac{1}{\alpha}+\gamma)^{-1}$. Similarly,  calculations on a class
  of polynomially intermittent maps which  preserve Lebesgue measure, extensively investigated in
  ~\cite{Lebesgue_IM} suggest that the addition formula $(\frac{1}{\alpha}+\gamma)^{-1}$ is not general, but 
  depends upon how close orbits come to the cusp or indifferent fixed point before returning to the inducing set. 
  It would be good to have a clear  understanding of any general principles which are involved. 
  
  Theorem~\ref{distribution} shows that under strong mixing assumptions on $(T,X,\mu)$ if $\phi$ is in the domain of 
  attraction of a  stable law of index $\alpha$,
  then $\frac{1}{n^{\frac{1}{\alpha}+\epsilon}} [\sum_{j=0}^{n}\phi\circ T^j -c_n]\dto 0$ (in distribution). Is there a corresponding almost sure convergence result under suitable mixing assumptions? Some related partial results
  are given in~\cite{Merlevede_IM,Balint_Terhesiu}. The case of iid random variables of stable index $\alpha$ is 
  given in~\cite{Chover}.

\section{Appendix}\label{sec.appendix}

\subsection{A convergence in distribution result}

We will use a result in Billingsley~\cite[Exercise
10.3]{Billingsley1999},\cite[p. 156]{Doob1953}; (see also~\cite[Page
1228]{Serfling1970}).

\begin{prop}[{\cite[Exercise 10.3]{Billingsley1999}}]\label{Billingsley}
  Let $\nu\ge 1$ and $\gamma\ge 1$. Suppose $X_i$ is a sequence of random
  $L^2$ variables on a probability space
  and there exist non-negative numbers $\{u_i\}$ such that for all
  $0 \le a \le a+m \le n$, 
  \[
    \mu \left(|\sum_{i=a}^{a+m} X_i |^{\nu}\right)\le (\sum_{i=a}^{a+m}
    u_i)^{\gamma}
  \]
  Then
  \[
    \mu \left( \max_{0\le j\le m} |\sum_{i=a}^{a+j} X_i |^{\nu}\right) \le
    (\log_2 4m)^{\nu} (\sum_{i=a}^{a+m} u_i)^{\gamma} \qquad \text{for
      $0 \le a \le a+m \le n$.}
  \]
\end{prop}

\subsection{Proof of Corollary~\ref{distribution-corollary}}\ \null

\begin{proof}
  
  Given $\epsilon >0$ small, we describe the approximations of
  Theorem~\ref{distribution} in Lipschitz  versus Lipschitz in dimension $D=1$ (other dimensions are
  similar with obvious modifications).

   It suffices to consider the case $p=1$, $C_1=1$. 
  Then
  \[
    \psi_n(x)=d(x,x_1)^{-\frac{1}{\alpha}} \cdot
    \one_{\{d(x,x_1)>n^{-(1+\alpha\epsilon)}\}}
  \]

  Let $g_n$ be $\psi_n$ with the addition (in cross-section) of a triangle
  where $\psi_n$ is truncated from $\phi$. Take the base of this triangle
  to be $n^{-5-2/\alpha}$, so the area of the triangle is
  $O(n^{-4-1/\alpha})$. Therefore
  $\|\psi_n - g_n\|_{L^1}=O(n^{-4-1/\alpha})$.

  The $L^\infty$ bound is clearly satisfied, and the Lipschitz norm of
  $g_n$ is given by the slope of the added line, so satisfies the
  requirement as well.

  Finally, let $f_n=g_n-\mu(g_n)$.
  \end{proof}

\subsection{Proof of Theorem~\ref{distribution}}\ \null

\begin{proof}

Throughout our proof $C$ will be used as a constant which may vary in
places but is independent of indices. Recall we assume $\phi$ satisfies
\begin{equation*}
    \mu(|\phi|> x) \le C x^{-\alpha}
  \end{equation*}

Let $\epsilon>0$. Truncate $\phi$ by defining
$\psi_n:=\phi \one_{A_n}$ with $\mu (A_n^c)\le Cn^{-1-\alpha\epsilon}$, $\|\phi\one_{A_n^c} \|_{L^1(\mu)}\le Cn^{\frac{1}{\alpha}-1}$ 
and $\|\psi_n\|_{\infty}\le Cn^{\frac{1}{\alpha}+\epsilon}$.
Let $\bar{\psi}_n=\psi_n-\mu (\psi_n )$.

Define $\bar{\phi}=\phi-\mu (\phi)$. Note that
$\mu(|\phi|>n^{\frac{1}{\alpha }+\epsilon})\le
(n^{\frac{1}{\alpha}+\epsilon})^{-\alpha}=n^{-1-\alpha\epsilon}$.

Note that $\mu (|\phi| \one_{|\phi|>t } )=t\mu (|\phi |>t) +\int_t^{\infty} \mu (|\phi|>s)ds \le Ct^{1-\alpha}$. Since
$ \mu (|\phi-\psi_n | )=\|\phi\one_{A_n^c}\|_{L^1(\mu)}$, by assumption

\begin{equation}\label{L1diff}
  \mu (|\phi-\psi_n | )    \le Cn^{\frac{1}{\alpha}-1}
\end{equation}

Thus 
\[
  n^{-(\frac{1}{\alpha}+\epsilon) } \sum_{j=1}^n
  (\mu (\phi )-\mu ( \psi_n ))\le Cn^{-\epsilon}
  \]
and so
\begin{equation}\label{means}
n^{-(\frac{1}{\alpha}+\epsilon)} \sum_{j=1}^n (\mu (\phi )-\mu (\psi_n ))\to 0
\end{equation}

By assumption
\[
\mu (\phi \not = \psi_n) \le C n^{-1-\alpha\epsilon}
\]
Thus
\begin{equation*}
  \mu \left(\max_{j\le n} \left|\sum_{i=0}^j\phi \circ F^i\right| \not = \max_{j\le n}
  \left|\sum_{i=0}^j \psi_n \circ F^i\right| \right)\le n\mu (\phi \not = \psi_n) \le C n^{-\alpha \epsilon}
\end{equation*}
Together with \eqref{means}, this implies that
\begin{equation}\label{error-maxima}
 n^{-(\frac{1}{\alpha}+\epsilon)} \left|\, \max_{j\le n} \left|\sum_{i=0}^j \bar\phi \circ F^i \right|- \max_{j\le n}
  \left|\sum_{i=0}^j \bar\psi_n \circ F^i\right| \,\right| \dto 0.
\end{equation}

We will show that 
\begin{equation}\label{eq:max-estimate}
  n^{-(\frac{1}{\alpha} +\epsilon) }\max_{j\le n}\left|\sum_{i=0}^{j}
  \bar{\psi}_n \circ F^i\right| \dto 0 \text{ as $n\to\infty$}
\end{equation}
which together with Equation~\eqref{error-maxima} implies that
\[
  n^{-(\frac{1}{\alpha}+\epsilon )}\max_{j\le n} \left|\sum_{i=0}^{j}
  \bar{\phi}\circ F^i\right| \dto 0
\]
in distribution, as claimed.

For \eqref{eq:max-estimate} we use the result of Billingsley,
Proposition~\ref{Billingsley}, with $\nu=2$, $\gamma=1$,
$X_i = \bar{\psi}_n\circ F^i$ and
\[
  u_i=C n^{2/\alpha -1+\epsilon(2-\alpha)}\log n
\]

By stationarity,
\[
  \mu  [ (\sum_{j=a}^{a+m-1} \bar{\psi}_n \circ F^j)^2]
  = \mu [(\sum_{j=1}^{m} \bar{\psi}_n \circ F^j)^2]= \sum_{j=1}^m\mu [
  \bar{\psi}_n^2\circ F^j ] + 2\sum_{i=1}^m\sum_{j=i+1}^m\mu [\bar{\psi}_n
  \circ F^i \bar{\psi}_n \circ F^j]
\]

Now $\mu (|\phi|^2 \one_{|\phi|\le t}) \le 2\int_0^t s\mu(|\phi|>s) ds \le Ct^{2-\alpha}$.

Since $\mu (\psi_n^2) \le \mu (|\phi|^2 \one_{|\phi |\le n^{\frac{1}{\alpha}+\epsilon} )}$ we have 
\begin{equation}\label{L2}
  \mu (\psi^2_n )\le  C n^{2/\alpha-1+\epsilon(2-\alpha)} 
\end{equation}
and, since $\mu (\phi )$ is finite, we also have that
\begin{equation}\label{L2-normalized}
  \mu (\bar{\psi}_n^2 ) = \mu (\psi_n^2 ) - (\mu (\psi_n ))^2 \le C
  n^{\frac{2}{\alpha}-1+\epsilon(2-\alpha)}
\end{equation}
Therefore
\[
  \sum_{j=1}^m \mu [ (\bar{\psi}_n\circ F^j)^2] \le
  Cmn^{2/\alpha-1+\epsilon(2-\alpha)}
\]

Using that $F$ preserves $\mu$, the double sum above is
\[
  2\sum_{i=1}^m\sum_{j=i+1}^m\mu [\bar{\psi}_n \circ F^{j-i} \bar{\psi}_n]
\]
and can be approximated by the sum with $f_n$:
\[
  \sum_{i=1}^m\sum_{j=i+1}^m|\mu [\bar{\psi}_n \circ F^{j-i}
  \bar{\psi}_n]-\mu [f_n \circ F^{j-i} f_n]|
\]
\[
  \le \sum_{i=1}^m \sum_{j=i+1}^m
  (\|\bar{\psi}_n\|_{\infty}+\|f_n\|_{\infty}) \|\bar{\psi}_n-f_n\|_1 \le C
  m^2 \|\bar{\psi}_n\|_{\infty} \|\bar{\psi}_n-f_n\|_1 \le C n^{-1}
\]

As a consequence of the exponential decay of correlations for $B_1$ versus
$B_2$
\[
  |\mu [f_n \circ F^{j-i} f_n]|\le C \theta^{j-i} n^{2\kappa}
\]
there exists
\[
  C_{n}=|\frac{2\kappa+4}{\log \theta}| \log n 
\]
such that if $j-i>C_n$ then
\[
\mu [ f_n \circ F^{j-i} f_n]\le n^{-3}
\]
and so
\[
  \sum_{i=1}^m \sum_{j-i>C_n } |\mu [ f_n \circ F^{j-i} f_n]|= O(n^{-1})
\]

It remains to estimate
\[
  \sum_{i=1}^m\sum_{j=i+1}^{i+C_n} \mu [\bar{\psi}_n\circ F^j \bar{\psi}_n
  \circ F^i] =\sum_{i=1}^m\sum_{k=1}^{C_n} \mu [\bar{\psi}_{n} \circ F^{k}
  \bar{\psi}_n]
\]
By the H\"older inequality and stationarity, using~\eqref{L2-normalized}
\[
  |\mu [\bar{\psi}_n \circ F^{k} \cdot \bar{\psi}_n]| \le
  \mu (\bar{\psi}_n^2 ) \le C n^{2/\alpha-1+\epsilon(2-\alpha)}
\]
so
\[
  |\sum_{i=1}^m\sum_{k=1}^{C_n } \mu [\bar{\psi}_n \circ F^{k} \cdot
  \bar{\psi}_n]| \le C\cdot C_n m n^{2/\alpha-1+\epsilon(2-\alpha)} \le C m\log n
  n^{2/\alpha-1+\epsilon(2 -\alpha)} 
\]
Collecting all the estimates, we obtain that
\[
  \mu [(\sum_{j=a}^{a+m} \bar{\psi}_n \circ F^j)^2] \le C m n^{2/\alpha
    -1+\epsilon(2-\alpha)}\log n \le \sum_{j=a}^{a+m} u_j
\]
with
\[
  u_i=C n^{2/\alpha -1+\epsilon(2-\alpha)}\log n
\]

Applying Proposition~\ref{Billingsley} with $\nu=2$, $\gamma=1$ we obtain
 \[
   \mu \left[\max_{j\le n}\left|\sum_{i=0}^{j} \bar{\psi}_n \circ F^i (x)\right|^2\right]\le
   C(\log_2 4n)^2 (\log n) n^{2/\alpha+\epsilon(2-\alpha)}
\]
The claim \eqref{eq:max-estimate} follows now by Chebyshev's inequality,
since
\[
  \mu \left(\max_{j\le n}\left|\sum_{i=0}^{j} \bar{\psi}_n \circ F^i (x)\right|>\delta
  n^{\frac{1}{\alpha}+\epsilon} \right) 
\]
\[
  \le \delta^{-2} n^{-\frac{2}{\alpha}-2\epsilon} \mu \left[\max_{j\le
    n}\left|\sum_{i=0}^{j} \bar{\psi}_n \circ F^i (x)\right|^2\right] =O(\delta^{-2} (\log
  n)^3 n^{-\alpha\epsilon})
\]
\end{proof}

\subsection{A result of Gou\"ezel}\label{gouezel_lifting}

We use the following result of Gou\"{e}zel~\cite[Theorem
4.6]{Gouezel_Doubling}:

\begin{prop}\label{gouezel-lifting}
Let $(T,X,\mu)$  be an ergodic probability preserving map, let $\alpha (n)$ and  and $B_n$ be two sequences of integers which are regularly varying with positive indices.  Let $A_n\in \R$ and let $Y \subset X$ be a subset with positive measure. We will denote by $\mu_Y (.):= \frac{\mu |_{Y}  }{\mu (Y)}$ the induced probability measure.
Let $R: Y \to \bN$ be the return time of $T$  to $Y$ and $F =T^{R} : Y \to Y$ be the induced map. Define $\bar{R}=\int_Y R d\mu_Y= \frac{1}{\mu (Y)}$. Consider a measurable function $\phi: X\to \R$ and define $\Phi : Y \to \R$ by $\Phi (y)=\sum_{j=0}^{R(y)-1} \phi\circ T^j$.
Define $S_n (\Phi)=\sum_{j=0}^{n-1} \Phi\circ F^j$. Assume that 
\[
\frac{ S_n (\Phi)-A_n}{B_n}
\]
converges in distribution (with respect to $\mu_Y$) to a random variable $S$.

  Additionally assume that either:
  \begin{enumerate}
  \item [(a)] $\frac{\sum_{j=0}^n R\circ F^j- n\bar{R}}{\alpha(n)}$ tends
    in probability to zero and
    $\max_{0\le k\le \alpha (n)} \frac{|S_k (\Phi)|}{B_n}$ is tight
  \item []
  \item [or]
  \item []
  \item [(b)] $\frac{\sum_{j=0}^n R\circ F^j- n\bar{R}}{\alpha(n)}$ is tight and
    $\max_{0\le k\le \alpha (n)} \frac{|S_k (\Phi)|}{B_n}$ tends in probability
    to zero.
  \end{enumerate}

  Then 
  \[
    \big(\sum_{j=0}^{n-1} \phi\circ T^j -A_{\left \lfloor n\mu(Y)\right
      \rfloor}\big)/B_{\left \lfloor n\mu (Y)\right \rfloor }
  \]
  converges in distribution (with respect to $\mu$) to $S$.

\end{prop}

\subsection{Data availability statement.}
	
This paper is proof based and no datasets were generated or analyzed during the current study.

\subsection{Conflict of interest statement.}

The authors declare that they have no conflict of interest.

\bibliographystyle{alpha} 

\bibliography{references}

\end{document}